\input ./style/arxiv-general.cfg
\documentclass[aap,MSNbibl,nameyear,dvips]{arximspdf}
\makeatletter
   \@ifpackageloaded{graphicx}{}{\usepackage{graphicx}}
\makeatother

\doi{10.1214/14-AAP1067} 
\volume{25}
\issue{6}
\pubyear{2015}
\firstpage{3047}
\lastpage{3094}
\docsubty{FLA}

\makeatletter
\newcommand{\rrvert}{\vert}
\newcommand{\rrVert}{\Vert}
\newcommand{\llvert}{\vert}
\newcommand{\llVert}{\Vert}
\newcommand{\IP}{\mathbb{P}}
\newcommand{\R}{\mathbb{R}}
\newcommand{\C}{\mathbb{C}}
\newcommand{\N}{\mathbb{N}}
\newcommand{\IS}{\mathbb{S}}
\newcommand{\IT}{\mathbb{T}}
\newcommand{\IA}{\mathbb{A}}
\newcommand{\ga}{\alpha}
\newcommand{\gb}{\beta}
\newcommand{\gd}{\delta}
\renewcommand{\gg}{\gamma}
\newcommand{\gk}{\kappa}
\newcommand{\gs}{\sigma}
\newcommand{\gO}{\Omega}
\newcommand{\cA}{\mathcal{A}}
\newcommand{\cB}{\mathcal{B}}
\newcommand{\cF}{\mathcal{F}}
\newcommand{\cH}{\mathcal{H}}
\newcommand{\cN}{\mathcal{N}}
\newcommand{\cY}{\mathcal{Y}}
\newcommand{\E}{\mathbb{E}} 
\newcommand{\Span}{\operatorname{span}}
\newcommand{\trace}{\operatorname{Tr}}
\newcommand{\SO}{\mathsf{SO}}
\let\Re\relax
\newcommand{\Re}{\operatorname{Re}}
\let\Im\relax
\newcommand{\Im}{\operatorname{Im}}
\newtheorem{lemma}{Lemma}[section]
\newtheorem{proposition}[lemma]{Proposition}
\newtheorem{theorem}[lemma]{Theorem}
\newtheorem{corollary}[lemma]{Corollary}
\newproclaim{remark}[lemma]{Remark}
\newproclaim{definition}[lemma]{Definition}
\newproclaim{assumption}[lemma]{Assumption}
\makeatother

\begin{document}
\begin{frontmatter}

\title{Isotropic Gaussian random fields on the sphere: Regularity,
fast simulation and stochastic partial~differential equations}
\runtitle{Isotropic GRFs on the sphere}

\begin{aug}
\author[A]{\fnms{Annika}~\snm{Lang}\corref{}\thanksref{T1}\ead[label=e1]{annika.lang@chalmers.se}}
\and
\author[B]{\fnms{Christoph} \snm{Schwab}\thanksref{T2}\ead[label=e2]{schwab@math.ethz.ch}}
\runauthor{A. Lang and Ch. Schwab}
\affiliation{Chalmers University of Technology, University of
Gothenburg and ETH Z\"urich,
and ETH Z\"urich}
\address[A]{Department of Mathematical Sciences\\
Chalmers University of Technology \& \\
University of Gothenburg\\
S--412 96 G\"oteborg\\
Sweden\\
and\\
Seminar f\"ur Angewandte Mathematik\\
ETH Z\"urich\\
R\"amistrasse 101\\
CH--8092 Z\"urich\\
Switzerland\\
\printead{e1}}
\address[B]{Seminar f\"ur Angewandte Mathematik\\
ETH Z\"urich\\
R\"amistrasse 101\\
CH--8092 Z\"urich\\
Switzerland\\
\printead{e2}}
\end{aug}
\thankstext{T1}{Supported in part by ERC AdG no.~247277 and by the
Knut and Alice Wallenberg foundation.}
\thankstext{T2}{Supported in part by ERC AdG no.~247277.}

\received{\smonth{5} \syear{2013}}
\revised{\smonth{5} \syear{2014}}

%
\begin{abstract}
Isotropic Gaussian random fields on the sphere are characterized by
Kar\-hunen--Lo\`eve expansions with respect to the spherical harmonic
functions and the angular power spectrum.
The smoothness of the covariance is connected to the decay of
the angular power spectrum and the relation to sample H\"older
continuity and sample differentiability of the random fields is discussed.
Rates of convergence of their finitely
truncated Kar\-hunen--Lo\`eve expansions in terms of the covariance
spectrum are established,
and algorithmic aspects of fast sample generation
via fast Fourier transforms on the sphere are indicated.
The relevance of the results on sample regularity
for isotropic Gaussian random fields and the corresponding lognormal
random fields on the sphere for several models from environmental
sciences is indicated.
Finally, the stochastic heat equation on the sphere driven by additive,
isotropic Wiener
noise is considered, and strong convergence rates for spectral
discretizations based on the spherical harmonic functions are proven.
\end{abstract}

%
\end{frontmatter}

\section{Introduction}\label{secintro}\label{sec1}
Sample regularity of Gaussian random fields (GRFs) on
subsets of Euclidean space is well studied, where
the spectral theory of these fields is used; see, for example,
\citeauthor{Yaglom2ndOrd} (\citeyear{Yaglom2ndOrd,YaglomI,YaglomII}).
However,
the general theory of second-order random fields as developed
in \citeauthor{Yaglom2ndOrd} (\citeyear{Yaglom2ndOrd,YaglomI,YaglomII})
requires a group structure on the space
of realizations. The\vspace*{1pt} (practically relevant) case
of GRFs indexed by the sphere, which we denote by~$\IS^2$ (and, more
generally, $\IS^{2n}$),
takes a special role with regard to invariance under
(topological) group actions [see, e.g.,~\citet{Megia07} and the references
there for a lucid discussion], so that the general results
in~\citet{Yaglom2ndOrd} do not apply directly.
Due to the relevance of GRFs on~$\IS^2$ in applications,
in particular in environmental modeling and cosmological data analysis
[cp.~\citet{MP11}],
it is of some interest to develop a theory
of sample regularity, stochastic partial differential equations and
their numerical analysis.
The contribution of some basic results with direct proofs
as well as the corresponding results on higher-dimensional spheres~$\IS^{d-1}$
is the purpose of the present paper.

Specifically, we derive the connection between the smoothness
of the covariance kernel of an isotropic GRF on~$\IS^2$ and the decay
of its angular power spectrum and characterize its $\IP$-a.s. sample
H\"older continuity and sample differentiability. Furthermore we
construct isotropic $Q$-Wiener processes using isotropic GRFs.
We solve the stochastic heat equation on~$\IS^2$ driven by isotropic
$Q$-Wiener noise with a series expansion with respect to the spherical
harmonic functions. We show that the convergence rate of the fully
discrete approximation scheme given by the truncation of the series
expansion depends only on the decay of the angular power spectrum and
that it is independent of the chosen space and time discretization.

The outline of this paper is as follows:
in Section~\ref{seciGRF} we recapitulate basic definitions
of isotropic GRFs on~$\IS^2$ and of the Kar\-hunen--Lo\`eve expansions
in spherical harmonic functions of these fields from~\citet{MP11}.
A characterization of
the decay of the angular power spectrum
of isotropic GRFs in terms of the regularity of the covariance kernel
in a scale of weighted Sobolev spaces on~$\IS^2$ is presented in
Section~\ref{secPowSpecDec}.
Section~\ref{secHoeldercont} contains a version of the
Kolmogorov--Chentsov theorem for random fields on~$\IS^2$, and
therefore sample H\"older continuity of random fields is addressed.
Sufficient conditions on the angular power spectrum
are presented for
$\IP$-a.s. sample H\"older continuity and differentiability of
isotropic GRFs.
In Section~\ref{secapproxiGRF} we approximate isotropic Gaussian
random fields by finite truncation of their Kar\-hunen--Lo\`eve
expansions. We discuss convergence rates of these approximations in
$p$th moment and in the $\IP$-a.s. sense.
The topic of Section~\ref{sechailstones} is the \hyperref[sec1]{Introduction} of the
practically important case of lognormal random fields. These are
crucial in a number of applications, in particular in meteorology
and in climate modeling. In this section, we give analogous results to
Section~\ref{secHoeldercont}; that is, sample regularity of lognormal
random fields in terms of H\"older continuity and differentiability is
addressed.
Finally, isotropic $Q$-Wiener processes are introduced in Section~\ref
{secstochheateqn}. We consider the stochastic heat equation on~$\IS^2$
driven by an isotropic $Q$-Wiener process and solve the stochastic
partial differential equation (SPDE) with spectral methods. We
approximate the solution by truncation of the derived spectral
representation and show convergence rates in $p$th moment as well as
$\IP$-almost surely.
These results are illustrated\vspace*{1pt} by numerical examples.
Although the main focus of the paper is the unit sphere $\IS^2$ due to
its relevance in applications, Sections~\ref{seciGRF}--\ref
{sechailstones} also include the corresponding results for
higher-dimensional spheres~$\IS^{d-1}$.

\section{Isotropic Gaussian random fields on the sphere}\label{seciGRF}
In this section we introduce isotropic Gaussian random fields and their
properties. We focus especially on Kar\-hunen--Lo\`eve expansions of
these random fields.
In doing so,
we follow closely the introduction of Gaussian random fields in
Chapter~5 of~\citet{MP11}.
We will first focus on Gaussian random fields on the unit sphere
embedded into $\R^3$ before we give a short review of Gaussian random
fields on unit sphere in arbitrary dimensions.
Throughout, we denote by $(\gO, \cA, \IP)$ a probability space and
write $\IS^2$ for the unit sphere in~$\R^3$,
that is,
\[
\IS^2 = \bigl\{ x \in\R^3, \llVert x\rrVert= 1\bigr\},
\]
where $\llVert \cdot\rrVert $ denotes the Euclidean norm.
Let $(\IS^2,d)$ be the compact metric space with the
geodesic metric given by
\[
d(x,y) = \arccos\langle x,y \rangle_{\R^3}
\]
for all $x,y \in\IS^2$.
We denote by $\cB(\IS^2)$ the Borel $\gs$-algebra of~$\IS^2$.

\begin{definition}
A $\cA\otimes\cB(\IS^2)$-measurable mapping $T\dvtx\Omega\times\IS^2
\rightarrow\R$ is called
a \emph{real-valued random field} on the unit sphere.

The random field $T$ is called \emph{strongly isotropic} if for all $k
\in\N$, $x_1,\ldots, x_k \in\IS^2$
and for $g \in\SO(3)$, the multivariate random variables
$(T(x_1), \ldots, T(x_k))$ and $(T(gx_1), \ldots, T(gx_k))$
have the same law, where $\SO(3)$ denotes the group of rotations
on~$\IS^2$.

It is called \emph{$n$-weakly isotropic} for $n \ge2$
if $\E(\llvert T(x)\rrvert ^n) < + \infty$ for all $x \in\IS^2$ and
if for $1 \le k \le n$, $x_1, \ldots,x_k \in\IS^2$ and $g \in\SO(3)$
\[
\E\bigl(T(x_1) \cdots T(x_k)\bigr) = \E\bigl(T(g
x_1) \cdots T(g x_k)\bigr).
\]

Furthermore it is called \emph{Gaussian} if for all $k \in\N$, $x_1,
\ldots, x_k \in\IS^2$
the multivariate random variable $(T(x_1),\ldots,T(x_k))$ is multivariate
Gaussian distributed; that is, $\sum_{i=1}^k a_i T(x_i)$ is a normally
distributed random variable
for all $a_i \in\R$, $i=1,\ldots,k$.
\end{definition}
In what follows, we focus on real-valued random fields.
Similarly to a Gaussian random field (GRF for short) on~$\R^d$, $d \in
\N$,
a GRF on~$\IS^2$ has the following property
proven, for example, in Proposition~5.10(3) in~\citet{MP11}.

\begin{proposition}
Let $T$ be a GRF on~$\IS^2$.
Then $T$ is strongly isotropic if and only if $T$ is $2$-weakly isotropic.
\end{proposition}

A key role in our analysis and simulation
of isotropic GRFs on $\IS^2$ is
taken by their Kar\-hunen--Lo\`eve expansions.
To introduce Kar\-hunen--Lo\`eve expansions of isotropic GRFs (and the
corresponding $Q$-Wiener processes on~$\IS^2$ in the formulation
of SPDEs on~$\IS^2$ in Section~\ref{secstochheateqn}), we first
define spherical harmonic functions on~$\IS^2$. 
We recall that the \emph{Legendre polynomials}
$(P_\ell, \ell\in\N_0)$ are, for example,
given by Rodrigues's formula [see, e.g., \citet{Szego}]
\[
P_\ell(\mu):= 2^{-\ell} \frac{1}{\ell!}
\frac{\partial^\ell}{\partial\mu
^\ell} \bigl(\mu^2 -1\bigr)^\ell
\]
for all $\ell\in\N_0$ and $\mu\in[-1,1]$.
The Legendre polynomials define the \emph{associated Legendre
functions} $(P_{\ell m}, \ell\in\N_0,m=0,\ldots,\ell)$ by
\[
P_{\ell m}(\mu):= (-1)^m \bigl(1-\mu^2
\bigr)^{m/2} \frac{\partial^m}{\partial\mu^m} P_\ell(\mu)
\]
for $\ell\in\N_0$, $m = 0,\ldots,\ell$ and $\mu\in[-1,1]$.
Here and throughout we do not separate indices for doubly subscripted
functions and coefficients by a comma with the understanding
that the reader will recognize double indices as such.
With this in mind, we further introduce
the \emph{surface spherical harmonic functions}
$\cY:= (Y_{\ell m}, \ell\in\N_0, m=-\ell, \ldots, \ell)$
as mappings
$Y_{\ell m}\dvtx [0,\pi] \times[0,2\pi) \rightarrow\C$,
which are given by
\[
Y_{\ell m}(\vartheta, \varphi):= \sqrt{\frac{2\ell+ 1}{4\pi}
\frac{(\ell-m)!}{(\ell+m)!}} P_{\ell m}(\cos\vartheta) e^{im\varphi}
\]
for $\ell\in\N_0$, $m = 0,\ldots, \ell$ and $(\vartheta,\varphi)
\in[0,\pi] \times[0,2\pi)$ and by
\[
Y_{\ell m}:= (-1)^m \overline{Y_{\ell-m}}
\]
for $\ell\in\N$ and $m=-\ell, \ldots,-1$.
By the Peter--Weyl theorem [see,\vspace*{1pt} e.g., Proposition~3.29 in~\citet{MP11}],
$\cY$ is an orthonormal basis of $L^2(\IS^2;\C)$
which we abbreviate by $L^2(\IS^2)$.
Every real-valued function~$f$ in~$L^2(\IS^2)$
admits the spherical harmonics series expansion
\[
f = \sum_{\ell=0}^\infty\sum
_{m=-\ell}^\ell f_{\ell m} Y_{\ell m}
\]
and the coefficients satisfy [cp., e.g., Remark~3.37 in~\citet{MP11}]
\[
f_{\ell m} = (-1)^m \overline{f_{\ell-m}};
\]
that is, $f$ can be represented in $L^2(\IS^2)$ by the series expansion
\[
f = \sum_{\ell=0}^\infty\Biggl(
f_{\ell0} Y_{\ell0} + 2 \sum_{m=1}^\ell
(\Re f_{\ell m} \Re Y_{\ell m} - \Im f_{\ell
m} \Im
Y_{\ell m} ) \Biggr).
\]
In what follows we set for $y \in\IS^2$
\[
Y_{\ell m}(y):= Y_{\ell m}(\vartheta,\varphi),
\]
where $y = (\sin\vartheta\cos\varphi, \sin\vartheta\sin\varphi,
\cos\vartheta)$; that is, we identify (with a slight abuse of notation)
Cartesian and angular coordinates of the point $y\in\IS^2$.
Furthermore we denote by $\gs$ the \emph{Lebesgue measure on the sphere}
which admits the representation
\[
d\gs(y) = \sin\vartheta \,d\vartheta \,d\varphi
\]
for $y \in\IS^2$, $y = (\sin\vartheta\cos\varphi, \sin\vartheta
\sin\varphi, \cos\vartheta)$.

We define the \emph{spherical Laplacian}, also called
\emph{Laplace--Beltrami operator}, in terms of spherical coordinates
similarly to Section~3.4.3 in~\citet{MP11}
by
\[
\Delta_{\IS^2}:= (\sin\vartheta)^{-1} \frac{\partial}{\partial\vartheta}
\biggl( \sin\vartheta\frac{\partial}{\partial\vartheta} \biggr) + (\sin
\vartheta)^{-2}
\frac{\partial^2}{\partial\varphi^2}.
\]
It is well known [see, e.g., Theorem~2.13 in~\citet{M98}]
that the spherical harmonic functions~$\cY$
are the eigenfunctions of~$\Delta_{\IS^2}$ with
eigenvalues $(-\ell(\ell+1), \ell\in\N_0)$, that is,
\[
\Delta_{\IS^2} Y_{\ell m} = - \ell(\ell+1) Y_{\ell m}
\]
for all $\ell\in\N_0$, $m = -\ell, \ldots,\ell$.
Furthermore it is shown in Theorem~2.42 of~\citet{M98} that
$L^2(\IS^2)$
has the direct sum decomposition
\[
L^2\bigl(\IS^2\bigr) = \bigoplus
_{\ell=0}^\infty\cH_\ell\bigl(
\IS^2\bigr),
\]
where the spaces $(\cH_\ell, \ell\in\N_0)$ are
spanned by spherical harmonic functions
\[
\cH_\ell\bigl(\IS^2\bigr):= \Span\{ Y_{\ell m}, m
= -\ell,\ldots,\ell\};
\]
that is, $\cH_\ell(\IS^2)$ denotes the space of eigenfunctions
of~$\Delta_{\IS^2}$
that correspond to the eigenvalue $-\ell(\ell+1)$ for $\ell\in\N_0$.

The significance of the spherical harmonic functions
lies in the fact that every $2$-weakly isotropic
random field admits a convergent
Kar\-hunen--Lo\`eve expansion.
The following result, which is proven in Theorem~5.13 in~\citet
{MP11} and a version of the Peter--Weyl theorem, makes this precise.

\begin{theorem}\label{thmKLexpansion}
Let $T$ be a $2$-weakly isotropic random field on $\IS^2$;
then the following statements hold true:
\begin{longlist}[(3)]
\item[(1)] $T$ satisfies $\IP$-almost surely
\[
\int_{\IS^2} T(x)^2 \,d\gs(x) < + \infty.
\]
\item[(2)] $T$ admits a Kar\-hunen--Lo\`eve expansion
%
\begin{equation}
\label{eqKLRF} T = \sum_{\ell=0}^\infty\sum
_{m=-\ell}^\ell a_{\ell m}
Y_{\ell m}
\end{equation}
with
\[
a_{\ell m} = \int_{\IS^2} T(y) \overline{Y_{\ell m}}(y)
\,d\gs(y)
\]
for $\ell\in\N_0$ and $m \in\{-\ell,\ldots, \ell\}$.
\item[(3)] Series expansion~(\ref{eqKLRF}) converges in $L^2(\gO
\times\IS^2;\R)$; that is,
\[
\lim_{L \rightarrow\infty} \E\Biggl( \int_{\IS^2}
\Biggl(T(y) - \sum_{\ell=0}^L \sum
_{m=-\ell
}^\ell a_{\ell m} Y_{\ell m}(y)
\Biggr)^2 \,d\gs(y) \Biggr) = 0.
\]
\item[(4)]
The series expansion~(\ref{eqKLRF}) converges in $L^2(\gO;\R)$ for
all $x \in\IS^2$; that is,
for all $x \in\IS^2$,
\[
\lim_{L \rightarrow\infty} \E\Biggl( \Biggl(T(x) - \sum
_{\ell=0}^L \sum_{m=-\ell}^\ell
a_{\ell m} Y_{\ell m}(x)\Biggr)^2 \Biggr) = 0.
\]
\end{longlist}
\end{theorem}

This result implies that every
$2$-weakly isotropic random field is an element of~$L^2(\gO;L^2(\IS^2))$.
For the efficient computational simulation of
$2$-weakly isotropic Gaussian random fields, which we will call in
the following just isotropic Gaussian random fields,
we will exploit special properties
of the random coefficients $\IA:= (a_{\ell m}, \ell\in\N_0, m=-\ell
,\ldots,\ell)$.
It turns out that the properties are similar to those of invariant
GRFs on the torus with Fourier series expansions; see, for example,
\citet{LP11}.
First of all we have by Remark~6.4, Proposition~6.6 and equation~(6.6)
in~\citet{MP11}
the following lemma.

\begin{lemma}\label{lempropalm}
Let $T$ be a strongly isotropic random field on~$\IS^2$ with Karhunen--Lo\`eve coefficients~$\IA$.
The elements of the sequence~$\IA$ are, except for $a_{00}$,
centered random variables, that is, $\E(a_{\ell m}) = 0$
for all $\ell\in\N$ and $m=-\ell,\ldots,\ell$.
Furthermore there
exists a sequence
$(A_\ell, \ell\in\N_0)$ of nonnegative real numbers
such that
\[
\E(a_{\ell_1 m_1} \overline{a_{\ell_2 m_2}}) = A_{\ell_1}
\gd_{\ell_1 \ell_2} \gd_{m_1 m_2}
\]
for $\ell_1,\ell_2 \in\N$ and $m_i = -\ell_i,\ldots,\ell_i$, $i=1,2$,
where $\gd_{n m} = 1$ if $n=m$ and zero otherwise.
For the first element $a_{00}$, it holds that
%
\[
\E(a_{00} \overline{a_{\ell m}}) = \bigl(A_0 +
\E(a_{00})^2\bigr) \gd_{0\ell} \gd_{0m}.
\]
The sequence $(A_\ell, \ell\in\N_0)$ is called
the \emph{angular power spectrum of~$T$}.

The random variables $a_{\ell m}$ and $a_{\ell-m}$
satisfy for $\ell\in\N$ and $m=1,\ldots,\ell$ that
\[
a_{\ell m} = (-1)^m \overline{a_{\ell-m}}.
\]
\end{lemma}

In the case of interest in this manuscript that $T$ is an
isotropic GRF, Theorem~6.12 in~\citet{MP11}
implies that
$\IA_+:= (a_{\ell m}, \ell\in\N_0, m = 0,\ldots,\ell)$
is a sequence of independent, complex-valued, Gaussian random variables.
By Proposition~6.8 in~\citet{MP11},
the elements of $\IA_+$ for $m \neq0$ satisfy
that $\Re a_{\ell m}$ and $\Im a_{\ell m}$
are symmetric random variables that are equal in law, uncorrelated,
that is, $\E(\Re a_{\ell m} \Im a_{\ell m})=0$, and
that have variance
\[
\E\bigl((\Re a_{\ell m})^2\bigr) = \E\bigl((\Im
a_{\ell m})^2\bigr) = A_\ell/2.
\]
By Lemma~\ref{lempropalm},
all elements of $\IA\setminus\IA_+$ 
can be obtained from $\IA_+$ via
\[
\Re a_{\ell m} = (-1)^m \Re a_{\ell-m}, \qquad\Im
a_{\ell m} = (-1)^{m+1} \Im a_{\ell-m}
\]
for $\ell\in\N$ and $m = -\ell, \ldots,-1$.
Furthermore we deduce from Propositions~6.11, 6.6 and
equation~(6.12) in~\citet{MP11} and
from Lemma~\ref{lempropalm} above
that
$\Re a_{\ell0}$ is $\cN(0,A_\ell)$ distributed; that is,
it is normally distributed with mean zero and variance~$A_\ell$, and
$\Im a_{\ell0} = 0$ for $\ell\in\N$ and that $\Re a_{0 0}$
is $\cN(\E(T)2\sqrt{\pi},A_0)$ distributed while $\Im a_{0 0} = 0$.

So, in conclusion, we have the following corollary.

\begin{corollary}\label{corKLexpiGRF}
Let $T$ be a $2$-weakly isotropic Gaussian random field on~$\IS^2$.
Then $T$ admits the Kar\-hunen--Lo\`eve expansion
\[
T = \sum_{\ell=0}^\infty\sum
_{m=-\ell}^\ell a_{\ell m}
Y_{\ell m},
\]
where $(Y_{\ell m}, \ell\in\N_0, m= - \ell, \ldots, \ell)$
is the sequence of spherical harmonic functions, and the sequence
$\IA:= (a_{\ell m}, \ell\in\N_0, m= - \ell, \ldots, \ell)$
is a sequence of complex-valued, centered, Gaussian random variables
with the following properties:
\begin{longlist}[(3)]
\item[(1)] $\IA_+:= (a_{\ell m}, \ell\in\N_0, m = 0,\ldots,\ell
)$ is a
sequence of independent, complex-valued Gaussian random variables.
\item[(2)] The elements of $\IA_+$ with $m>0$ satisfy
$\Re a_{\ell m}$ and $\Im a_{\ell m}$ are independent and $\cN
(0,A_\ell/2)$ distributed.
\item[(3)]
The elements of $\IA_+$ with $m=0$ are real-valued, and
the elements $\Re a_{\ell0}$ are $\cN(0,A_\ell)$
distributed for $\ell\in\N$ while $\Re a_{00}$ is $\cN(\E(T)2\sqrt
{\pi},A_0)$ distributed.
\item[(4)] The elements of~$\IA$ with $m <0$ are deduced from those
of $\IA_+$
by the formulas
\[
\Re a_{\ell m} = (-1)^m \Re a_{\ell-m}, \qquad\Im
a_{\ell m} = (-1)^{m+1} \Im a_{\ell-m}.
\]
\end{longlist}
\end{corollary}

Rather than
the specific case of~$\IS^2$, which is mainly relevant in applications,
we can also consider $\IS^2$ as a particular instance
of the unit sphere $\IS^{d-1}:= \{x \in\R^d, \llVert x\rrVert _{\R^d}
= 1\}$
embedded into $\R^d$ for some $d \ge2$.
%
The
angular distance~$d$ of two points $x$ and $y$ on~$\IS^{d-1}$ is given
in the same way as on~$\IS^2$ by $d(x,y) = \arccos\langle x,y\rangle
_{\R^d}$. Let us denote by $(S_{\ell m}, \ell\in\N_0, m=1,\ldots,h(\ell
,d))$ the spherical harmonics on~$\IS^{d-1}$, where
%
\[
h(\ell,d) = (2\ell+ d - 2) \frac{(\ell+d-3)!}{(d-2)! \ell!}.
\]
Using the framework of~\citet{Y83}, we call a $\cB(\IS
^{d-1})\times\cF$-measurable random field~$T$ on~$\IS^{d-1}$
isotropic if $\E(T(x))$ is constant for all $x \in\IS^{d-1}$,
without loss of generality $0$ and if the kernel of the covariance
$k_T(x,y) = \E(T(x)T(y))$ is given by a function of the distance
$d(x,y)$; that is, the distribution of the random field is invariant
under rotations. Then $T$ is mean square continuous by~\citet
{MP13}. It is shown in Section~5.1 in~\citet{Y83} that this
implies that $T$ admits a Karhunen--Lo\`eve expansion
\[
T(x) = \sum_{\ell=0}^\infty\sum
_{m=1}^{h(\ell,d)} a_{\ell m} S_{\ell m}(x),
\]
where $(a_{\ell m}, \ell\in\N_0,m=1,\ldots,h(\ell,d))$ is a
sequence of random variables that satisfy
\[
\E(a_{\ell m}) = 0, \qquad\E(a_{\ell m} a_{\ell' m'}) =
A_\ell\delta_{\ell\ell'} \delta_{m m'}
\]
for $\ell\in\N_0$ and $m=1,\ldots,h(\ell,d)$ and
\[
\sum_{\ell=0}^\infty A_\ell h(
\ell,d) < + \infty.
\]
The series converges with probability one and in $L^2(\gO;\R)$
as well as in $L^2(\gO;L^2(\IS^{d-1}))$.
If we assume further that $T$ is Gaussian, then the
random variables $(a_{\ell m}, \ell\in\N_0,m=1,\ldots,h(\ell,d))$
are independent,
and the convergence results extend to~$L^p(\gO;\R)$ and $L^p(\gO;L^2(\IS
^{d-1}))$, $p\ge1$.
Denoting by $(A_\ell, \ell\in\N_0)$ the angular power spectrum
for $\IS^{d-1}$ in analogy to what was done for~$\IS^2$,
there hold completely similar
properties for Gaussian isotropic random fields on~$\IS^{d-1}$.
%
\section{Decay of the angular power spectrum}\label{secPowSpecDec}
%

The error in a $\gk$-term truncation of the Kar\-hunen--Lo\`eve expansion
of an isotropic GRF~$T$ on~$\IS^2$
is closely related to the decay of the angular power spectrum
of~$T$.
As we show next, the decay of the angular power spectrum
is in turn characterized by the behavior of the
covariance kernel function that characterizes the isotropic
GRF~$T$.
Often the kernel function $k_T$ is prescribed
in applications.

To specify this relation,
we start with the definition of the kernel~$k_T$ of the
covariance of a centered isotropic Gaussian random field on~$\IS^2$
with prescribed
angular power spectrum $(A_\ell, \ell\in\N_0)$.
It is given for $x,y \in\IS^2$ by the formula
\begin{eqnarray*}
k_T(x,y)&:=& \E\bigl(T(x)T(y)\bigr)
\\
& = &\sum
_{\ell=0}^\infty A_\ell\sum
_{m=-\ell}^\ell Y_{\ell m}(x)
\overline{Y_{\ell m}}(y)
\\
& = &\sum_{\ell=0}^\infty A_\ell
\frac{2\ell+1}{4\pi} P_\ell\bigl(\langle x,y\rangle_{\R^3}\bigr).
\end{eqnarray*}
We observe that the covariance kernel~$k_T$
just depends on the inner product, respectively, the (spherical) distance.
Accordingly,
we denote by $k\dvtx[0,\pi] \rightarrow\R$ the kernel as a
function of the distance $r = d(x,y)$, that is,
\[
k(r):= \sum_{\ell=0}^\infty
A_\ell\frac{2\ell+1}{4\pi} P_\ell(\cos r)
\]
for $r \in[0,\pi]$.
A third way to look at the kernel is
in terms of the inner product $\langle x,y\rangle_{\R^3}$.
Therefore we define $k_I\dvtx[-1,1] \rightarrow\R$ by
\[
k_I(\mu):= k(\arccos\mu)
\]
for all $\mu\in[-1,1]$.
This implies overall for $x,y \in\IS^2$ that
\[
k_T(x,y) = k\bigl(d(x,y)\bigr) = k_I\bigl(\langle
x,y \rangle_{\R^3}\bigr).
\]
We will show that the regularity of the kernel is equivalent to the
weighted $2$-summability of
the angular power spectrum $(A_\ell, \ell\in\N_0)$, which can be
formalized in the framework of weighted Sobolev
spaces.

Therefore for $n \in\N_0$ let $H^n(-1,1) \subset L^2(-1,1)$ denote
the standard Sobolev spaces. We define
the function spaces $V^n(-1,1)$ as the closures
of $H^n(-1,1)$
with respect to the weighted norms
$\llVert \cdot\rrVert _{V^n(-1,1)}$ given by
\[
\llVert u \rrVert_{V^n(-1,1)}^2:= \sum
_{j=0}^n \llvert u \rrvert^2_{V^j(-1,1)}
,
\]
where for $j \in\N_0$ the seminorm $\llvert \cdot\rrvert
_{V^j(-1,1)}$ is defined by
\[
\llvert u \rrvert^2_{V^j(-1,1)}:= \int_{-1}^1
\biggl\llvert\frac{\partial^j}{\partial\mu^j} u(\mu) \biggr\rrvert^2
\bigl(1 -
\mu^2\bigr)^j \,d\mu.
\]
With this definition, $(V^n(-1,1), n \in\N_0)$ is a decreasing scale
of separable Hilbert spaces, that is,
\[
L^2(-1,1) = V^0(-1,1) \supset V^1(-1,1)
\supset\cdots\supset V^n(-1,1) \supset\cdots.
\]
By Ehrling's lemma the norm of $V^n(-1,1)$ is equivalent to the first
and the last element of the sum, that is,
\[
\llVert u \rrVert_{V^n(-1,1)}^2 \simeq\llVert u
\rrVert_{L^2(-1,1)}^2 + \llvert u \rrvert_{V^n(-1,1)}^2
\]
for all $u \in V^n(-1,1)$.
We will in the sequel not distinguish between these
norms by a separate notation.

In what follows we are deriving further equivalent norms of~$V^n(-1,1)$
in terms of summability of the spectrum. Therefore let us first observe
that any $u \in L^2(-1,1)$ can be expanded in the $L^2(-1,1)$
convergent Fourier--Legendre series
\[
u = \sum_{\ell=0}^\infty u_\ell
\frac{2\ell+1}{2} P_\ell
\]
with
\[
u_\ell:= \int_{-1}^1 u(x)
P_\ell(x) \,dx
\]
for all $\ell\in\N_0$.
Setting $ A_\ell:= 2\pi u_\ell$, we obtain that
\[
u = \sum_{\ell=0}^\infty A_\ell
\frac{2\ell+1}{4\pi} P_\ell;
\]
that is, $u$ is a valid kernel $k_I$.
So, instead of showing the equivalence of the regularity of the kernel
and the summability of the angular power spectrum, we can show an
isomorphism between the spaces $V^n(-1,1)$ and\vspace*{1pt} the weighted sequence
spaces $\ell_n:= \ell^2((\frac{2\ell+1}{2} (1+\ell^{2n}), \ell
\in\N_0))$, where $(\frac{2\ell+1}{2} (1+\ell^{2n}), \ell\in\N
_0)$ denotes the sequence of weights.
Since our goal is to extend this isomorphism to spaces $V^\eta(-1,1)$
with $\eta\notin\N_0$, we first extend the definition of the
weighted Sobolev spaces to nonintegers before we prove our main result.
We define for $n < \eta< n+1$ the interpolation space $V^\eta(-1,1)$
with the real method of interpolation in the sense of~\citet
{Triebel95} by
\[
V^\eta(-1,1):= \bigl( V^n(-1,1), V^{n+1}(-1,1)
\bigr)_{\eta-n,2}
\]
equipped with the norm
$\llVert \cdot\rrVert _{V^\eta(-1,1)}$
given by
\[
\llVert u \rrVert_{V^\eta(-1,1)}^2 = \int
_0^\infty t^{-2(\eta-n)} \bigl\llvert K(t,u)
\bigr\rrvert^2 \,\frac{dt}{t},
\]
where the $K$-functional is defined by
\[
K(t,u) = \inf_{u = v+w} \bigl( \llVert v \rrVert_{V^n(-1,1)}
+ t \llVert w \rrVert_{V^{n+1}(-1,1)} \bigr)
\]
for $t>0$.

The definition of the interpolation spaces~$\ell_\eta$ for $\eta
\notin\N_0$ is done similarly. The \emph{interpolation property} of
the spaces [see, e.g., step~4 in the proof of Theorem~1.3.3
in~\citet{Triebel95} or Proposition~2.4.1 in~\citet{T83}]
implies that the spaces $V^\eta(-1,1)$ and $\ell_\eta$ are
isomorphic for $\eta\in\R_+$ if this is true for $\eta\in\N_0$.

\begin{theorem}\label{thmBesovEquivNorm}
Let $u \in L^2(-1,1)$ and $\eta\in\R_+$ be given. Then $u \in V^\eta
(-1,1)$ if and only if
\[
\sum_{\ell=0}^\infty u_\ell^2
\frac{2\ell+1}{2} \bigl(1+\ell^{2\eta}\bigr) < + \infty;
\]
that is,
\[
\llVert u\rrVert_{V^\eta(-1,1)}^2 \simeq\sum
_{\ell=0}^\infty u_\ell^2
\frac{2\ell+1}{2} \bigl(1+\ell^{2\eta}\bigr)
\]
is an equivalent norm in~$V^\eta(-1,1)$.

For $k_I \in V^n(-1,1)$, $n \in\N_0$, this translates to the relation
that the sequence $(\ell^{n+1/2} A_\ell, \ell\ge n)$ is in~$\ell
^2(\N_0)$
if and only if $(1-\mu^2)^{n/2} \frac{\partial^n}{\partial\mu^n}
k_I(\mu)$,
$\mu\in(-1,1)$, is in~$L^2(-1,1)$; that is,
\[
\frac{1}{(4\pi)^2} \sum_{\ell\geq n} A_\ell^2
\frac{2\ell+1}{2} \ell^{2n} < +\infty
\]
if and only if
\[
\int_{-1}^1 \biggl\llvert
\frac{\partial^n}{\partial\mu^n} k_I(\mu) \biggr\rrvert^2 \bigl(1-
\mu^2\bigr)^n \,d\mu< +\infty.
\]
\end{theorem}

\begin{pf}
We divide the proof into two steps.
Let us assume first that the theorem is already proven for $\eta\in\N
_0$, that is, that $V^n(-1,1)$ is isomorphic to the weighted sequence
$\ell_n$ for all $n \in\N_0$. So let $n < \eta< n+1$ for some $n
\in\N_0$ be given and set $\theta:= \eta-n$.
Applying the interpolation theorem of Stein--Weiss [see, e.g.,
Theorem~5.4.1 in~\citet{BL76}], we get that the weights of $\ell
_\eta$ are given by
\begin{eqnarray*}
&& \biggl(\frac{2\ell+1}{2} \bigl(1+\ell^{2n}\bigr)
\biggr)^{1-\theta} \biggl(\frac{2\ell+1}{2} \bigl(1+\ell^{2(n+1)}\bigr)
\biggr)^\theta
\\
&& \qquad= \frac{2\ell+1}{2} \bigl(1+\ell^{2n}\bigr)^{1-\theta}
\bigl(1+\ell^{2(n+1)}\bigr)^\theta.
\end{eqnarray*}
It remains to show that this is equivalent to $\frac{2\ell+1}{2}
(1+\ell^{2\eta})$. But this follows immediately with the observation
that the function $x^p$, $p \in(0,1)$, is concave on $\R_+$ and
satisfies $(x+y)^p \ge2^{p-1} (x^p + y^p)$.

In the second step, let us prove the isomorphism of $V^n(-1,1)$ and
$\ell_n$ for $n \in\N_0$, which is the same as proving the second
formulation of the theorem.
Therefore let us first observe that by definition
\begin{eqnarray*}
&& \int_{-1}^1 \biggl\llvert\frac{\partial^n}{\partial\mu^n}
k_I(\mu) \biggr\rrvert^2 \bigl(1-\mu^2
\bigr)^n \,d\mu
\\
&& \qquad= \int_{-1}^1 \Biggl( \sum
_{\ell=0}^\infty A_\ell\frac{2\ell+1}{4\pi}
\frac{\partial^n}{\partial\mu^n} P_\ell(\mu) \Biggr)^2 \bigl(1-
\mu^2\bigr)^n \,d\mu
\\
&& \qquad= \sum_{\ell,\ell'=0}^\infty
A_\ell\frac{2\ell+1}{4\pi} A_{\ell'} \frac{2\ell'+1}{4\pi} \int
_{-1}^1 \biggl(\frac{\partial^n}{\partial\mu^n}
P_\ell(\mu) \biggr) \biggl(\frac{\partial^n}{\partial\mu^n} P_{\ell
'}(\mu)
\biggr) \bigl(1-\mu^2\bigr)^n \,d\mu.
\end{eqnarray*}
By $(P_\ell^{(\ga,\gb)}, \ell\in\N_0)$, we denote
the Jacobi polynomials given, for example, by Rodrigues's formula
\[
P_\ell^{(\ga,\gb)}(\mu):= \frac{(-1)^\ell}{2^\ell\ell!} (1-
\mu)^{-\ga} (1+\mu)^{-\gb} \frac{\partial^\ell}{\partial\mu^\ell} \bigl((1-
\mu)^{\ga} (1+\mu)^{\gb} \bigl(1-\mu^2
\bigr)^\ell\bigr)
\]
for $\ell\in\N_0$, $\ga,\gb> -1$ and $\mu\in[-1,1]$.
They satisfy that
\[
\frac{\partial}{\partial\mu} P_\ell^{(\ga,\gb)}(\mu) = \frac{1}{2} (
\ell+\ga+\gb+1) P_{(\ell-1)}^{(\ga+1,\gb+1)}(\mu).
\]
Since Legendre polynomials are particular instances of
Jacobi polynomials for $\ga= \gb=0$,
we conclude by recursion that
\[
\frac{\partial^n}{\partial\mu^n} P_\ell(\mu) = \frac{\partial
^n}{\partial\mu^n}
P_\ell^{(0,0)}(\mu) = \frac{(\ell+n)!}{2^n \ell!} P_{(\ell-n)}^{(n,n)}(
\mu)
\]
for every $n \le\ell$.
This implies that
\begin{eqnarray*}
&& \int_{-1}^1 \biggl(\frac{\partial^n}{\partial\mu^n}
P_\ell(\mu) \biggr) \biggl(\frac{\partial^n}{\partial\mu^n} P_{\ell
'}(\mu)
\biggr) \bigl(1-\mu^2\bigr)^n \,d\mu
\\
&& \qquad= \int_{-1}^1 \frac{(\ell+n)!}{2^n \ell!}
P_{(\ell
-n)}^{(n,n)}(\mu) \frac{(\ell'+n)!}{2^n \ell'!} P_{(\ell'-n)}^{(n,n)}(
\mu) (1-\mu)^n (1+\mu)^n \,d\mu
\\
&& \qquad= \gd_{\ell\ell'} \frac{2}{2\ell+1} \frac{(\ell
+n)!}{(\ell-n)!},
\end{eqnarray*}
where the last equation follows from the orthogonality of the
Jacobi polynomials [see, e.g.,~\citet{Szego}] and
\[
\int_{-1}^1 \bigl(P_{(\ell-n)}^{(n,n)}(
\mu) \bigr)^2 (1-\mu)^n (1+\mu)^n \,d\mu=
\frac{2^{2n+1}}{2\ell+1} \frac{\ell! \ell!}{(\ell-n)!
(\ell+n)!}.
\]
In conclusion we have shown that
\[
\int_{-1}^1 \biggl\llvert
\frac{\partial^n}{\partial\mu^n} k_I(\mu) \biggr\rrvert^2 \bigl(1-
\mu^2\bigr)^n \,d\mu= \sum_{\ell= n}^\infty
A_\ell^2 \frac{2\ell+1}{2 (4\pi)^2} \frac
{(\ell+n)!}{(\ell-n)!},
\]
since for $n>\ell$, the $n$th derivative of $P_\ell$ vanishes.
To finish the proof it remains to show that
for $n \le\ell$ there exist constants $c_1(n)$ and $c_2(n)$
such that
\[
c_1(n) \ell^{2n} \le\frac{(\ell+n)!}{(\ell-n)!} \le
c_2(n) \ell^{2n}.
\]
This follows from Stirling's inequalities 
%
\[
\sqrt{2\pi} \ell^{\ell+1/2} e^{-\ell} \le\ell! \le e \cdot
\ell^{\ell+1/2} e^{-\ell}
\]
for $\ell\in\N$
by writing
\[
\frac{(\ell+ n)^{\ell+n}}{(\ell-n)^{\ell-n}} = \ell^{\ell+n - (\ell-n)}
\frac{(1+n/\ell)^{\ell(1+n/\ell
)}}{(1-n/\ell)^{\ell(1-n/\ell)}}
\]
and by using the properties of the exponential function.
\end{pf}

So in conclusion we have shown that a necessary
and sufficient criterion for the weighted $2$-summability of
the angular power spectrum $(A_\ell, \ell\in\N_0)$
is the weighted square integrability
of the $n$th weak derivatives of $k_I$ with respect to the weight function
$(1-\mu^2)^n$. This is extended to nonintegers by the introduction of
weighted Sobolev spaces and the use of interpolation theory.
For more details on the interpolation results, we refer 
to the \hyperref[appinterpolspaces]{Appendix}.

So far we obtained results for GRFs on~$\IS^2$.
These can be extended to GRFs on~$\IS^{d-1}$, $d \ge2$, which we
briefly outline next.
We start with the definition of covariance kernels.
For a centered, mean square continuous, isotropic random field on~$\IS
^{d-1}$ with prescribed
angular power spectrum $(A_\ell, \ell\in\N_0)$, it is shown in
Section~I.5.1 of~\citet{Y83} that the kernel~$k_T$ of the
covariance is given by
\begin{eqnarray*}
k_T(x,y)&:=& \E\bigl(T(x)T(y)\bigr)
\\
& = &\sum
_{\ell=0}^\infty A_\ell\sum
_{m=1}^{h(\ell,d)} S_{\ell
m}(x)
\overline{S_{\ell m}}(y)
\\
& = &\frac{1}{\omega_d}\sum_{\ell=0}^\infty
A_\ell\frac{C_\ell
^{(d-2)/2}(\langle x,y\rangle_{\R^d})}{C_\ell^{(d-2)/2}(1)} h(\ell,d)
\end{eqnarray*}
for $x,y \in\IS^{d-1}$, where $\omega_d = 2\pi^{1+(d-2)/2}/\Gamma
(1+(d-2)/2)$ is the total area of~$\IS^{d-1}$, and $C_\ell^\eta$
denotes the Gegenbauer polynomial
\[
C_\ell^\eta(x):=\frac{\Gamma(\eta+1/2)\Gamma(d+2\eta)}{\Gamma
(2\eta)\Gamma(\ell+\eta+1/2)}P_\ell^{(\eta-1/2,\eta-1/2)}(x),
\]
which can be characterized in terms of Jacobi polynomials. Similarly
to~$\IS^2$, the representation of the kernel of the covariance extends
to the definitions of the other representations by
\[
k(r):= \frac{1}{\omega_d}\sum_{\ell=0}^\infty
A_\ell\frac{C_\ell
^{(d-2)/2}(\cos r)}{C_\ell^{(d-2)/2}(1)} h(\ell,d)
\]
for $r \in[0,\pi]$ for the kernel $k\dvtx[0,\pi] \rightarrow\R$ as a
function of the distance $r = d(x,y)$ and by
\[
k_I(\mu):= k(\arccos\mu)
\]
for all $\mu\in[-1,1]$ for the kernel $k_I\dvtx[-1,1] \rightarrow\R$ as
a function of the inner product $\langle x,y\rangle_{\R^d}$.
This implies for $x,y \in\IS^{d-1}$ that
\[
k_T(x,y) = k\bigl(d(x,y)\bigr) = k_I\bigl(\langle
x,y \rangle_{\R^d}\bigr).
\]
So overall, we have to extend our results from Legendre polynomials to
Gegenbauer polynomials, which leads to more generally weighted
$L^2(-1,1)$ and Sobolev spaces.

Therefore let us first observe that Stirling's inequalities imply that
for fixed~$d$
\[
\frac{h(\ell,d)}{C_\ell^{(d-2)/2}(1)} \simeq\ell,
\]
since
$C_\ell^{\eta}(1) = {\ell+ 2\eta-1\choose\ell}$
and
$h(\ell,d) = (2\ell+ d-2)\cdot(\ell+d-3)!/((d-2)!\ell!)$.
Furthermore we observe [cp.~\citet{BE53}, Section~10.9]
that
\[
\frac{\partial^n}{\partial\mu^n} C_\ell^\eta(\mu) = 2^n
\frac{\Gamma(\eta+n)}{\Gamma(\eta)} C_{\ell-n}^{\eta
+n}(\mu) \simeq
C_{\ell-n}^{\eta+n}(\mu)
\]
and that the Gegenbauer polynomials are orthogonal with respect to the
weighted $L^2(-1,1)$ norm given by
\[
\int_{-1}^1 C_\ell^\eta(
\mu) C_{\ell'}^\eta(\mu) \bigl(1-\mu^2
\bigr)^{\eta-1/2} \,d\mu= \delta_{\ell\ell'} \frac{\pi2^{1-2\eta}\Gamma
(\ell+ 2\eta
)}{\ell! (\ell+ \eta) \Gamma(\eta)^2} \simeq
\delta_{\ell\ell'} \ell^{2\eta-2},
\]
where the last step follows using again Stirling's inequalities and
assuming that $2\eta\in\N$.

Then the combination of these results leads to
\begin{eqnarray*}
\int_{-1}^1 \biggl( \frac{\partial^n}{\partial\mu^n}
k_I(\mu) \biggr)^2 \bigl(1-\mu^2
\bigr)^{(d-3)/2 + n} \,d\mu&\simeq&\sum_{\ell=0}^\infty
A_\ell^2 \ell^2 \ell^{d-4+2n}
\\
&=& \sum
_{\ell=0}^\infty A_\ell^2
\ell^{d-2+2n}.
\end{eqnarray*}
Defining the weighted Sobolev spaces $V^n(-1,1)$ for $n \in\N_0$ by
the completion of the standard Sobolev spaces with respect to the
weighted norms
\[
\int_{-1}^1 \biggl( \frac{\partial^n}{\partial\mu^n} u(\mu)
\biggr)^2 \bigl(1-\mu^2\bigr)^{(d-3)/2 + n} \,d\mu
\]
and using interpolation theory to define $V^\eta(-1,1)$ for positive
$\eta\notin\N_0$, we obtain the following generalization of
Theorem~\ref{thmBesovEquivNorm}.

\begin{theorem}
Let $u \in L^2(-1,1)$ and $\eta\in\R_+$ be given. Then $u \in V^\eta
(-1,1)$ if and only if
\[
\sum_{\ell=0}^\infty u_\ell^2
\ell^{d-2 +2\eta} < + \infty;
\]
that is,
\[
\llVert u\rrVert_{V^\eta(-1,1)}^2 \simeq\sum
_{\ell=0}^\infty u_\ell^2
\ell^{d-2 +2\eta}
\]
is an equivalent norm in~$V^\eta(-1,1)$.

For $k_I \in V^n(-1,1)$, $n \in\N_0$, this simplifies to the equivalence
that the sequence $(\ell^{d/2-1 + n} A_\ell, \ell\ge n)$ is in~$\ell
^2(\N_0)$
if and\vspace*{1pt} only if the function $(1-\mu^2)^{(d-3)/4+n/2} \frac{\partial
^n}{\partial\mu^n} k_I(\mu)$,
$\mu\in(-1,1)$, is in~$L^2(-1,1)$.
\end{theorem}

\section{Sample H\"older continuity and differentiability}\label{secHoeldercont}
Our analysis of GRFs via the Kar\-hunen--Lo\`eve expansion in
Section~\ref{seciGRF}
focused so far on mean square properties.
In this section we consider sample properties of isotropic GRFs
introduced in Section~\ref{seciGRF}.
Specifically,
we are interested how the $\IP$-almost sure H\"older continuity
of isotropic GRFs depends on the decay of the angular power
spectrum $(A_\ell, \ell\in\N_0)$ which is one possible
characterization of isotropic GRFs on~$\IS^2$ by Theorem~\ref
{thmKLexpansion} and Lemma~\ref{lempropalm}.
In the sequel we will frequently make use of a summability condition on
the angular power spectrum,
which we state in the following assumption.
%
\begin{assumption}[(Summability condition on the angular power
spectrum)]\label{asssummability-Al}
Assume that the angular power spectrum $(A_\ell, \ell\in\N_0)$ of
an iso\-tro\-pic Gaussian random field on~$\IS^{d-1}$ satisfies
for some $\gb>0$ that
\[
\sum_{\ell= 0}^\infty A_\ell
\ell^{d-2+\gb} < + \infty.
\]
%
\end{assumption}
The following lemma relates the decay of the
angular power spectrum to the H\"older continuity of the kernel~$k$
at zero, that is, to the H\"older continuity in mean square of the
corresponding random field.
The field is known to be mean square continuous by~\citet{MP13}, but
to derive exponents of sample H\"older continuity of (modifications of)
the random field, we need stronger results.

\begin{lemma}\label{lemHoeldercontkernelk}
Let $(A_\ell, \ell\in\N_0)$
be the angular power spectrum of an isotropic GRF on~$\IS^2$
which satisfies Assumption~\ref{asssummability-Al} with $d=3$
for some $\gb\in[0,2]$.
%
%
Then the corresponding kernel function $k$ satisfies that
there exists a constant $C_\gb$ such that for all $r \in[0,\pi]$
\[
\bigl\llvert k(0) - k(r)\bigr\rrvert\le C_\gb
r^\gb.
\]
\end{lemma}

\begin{pf}
We observe that $P_\ell(1) = 1$ for all $\ell\in\N_0$
and that the derivative of $P_\ell(x)$ is bounded by $P_\ell'(1)$ for
all $x \in[-1,1]$.
Therefore
\[
\bigl\llvert1-P_\ell(x)\bigr\rrvert= \biggl\llvert\int
_x^1 P_\ell'(y) \,dy
\biggr\rrvert\le\llvert1-x\rrvert\frac{\ell(\ell+1)}{2}.
\]
Furthermore we have that
\[
\bigl\llvert1-P_\ell(x)\bigr\rrvert\le2.
\]
This implies by interpolation that
\[
\bigl\llvert1-P_\ell(x)\bigr\rrvert\le\biggl(\llvert1-x\rrvert
\frac{\ell(\ell+1)}{2} \biggr)^\gg2^{1-\gg} \le2 \llvert1-x
\rrvert^\gg\bigl(\ell(\ell+1)\bigr)^\gg
\]
for all $\gg\in[0,1]$.
Using this estimate we obtain that
\begin{eqnarray*}
\bigl\llvert k(0) - k(r)\bigr\rrvert& \le&\sum_{\ell=0}^\infty
A_\ell\frac{2\ell+1}{4\pi} \bigl\llvert1-P_\ell(\cos r)\bigr
\rrvert
\\
& \le&(2\pi)^{-1} \llvert1-\cos r\rrvert^\gg\sum
_{\ell=0}^\infty A_\ell(2\ell+1)
\bigl(\ell(\ell+1)\bigr)^\gg,
\end{eqnarray*}
where the series converges if $\sum_{\ell= 0}^\infty A_\ell\ell
^{2\gg+ 1}$,
which holds by the made assumptions for all $\gg\le\gb/2$.
Finally we observe that
\[
\llvert1-\cos r\rrvert= \biggl\llvert\int_0^r
\sin x \,dx \biggr\rrvert\le r \sin r = r \int_0^r
\cos x \,dx \le r^2 \cdot1,
\]
which implies overall with the choice $\gb= 2\gg$ that
\[
\bigl\llvert k(0) - k(r)\bigr\rrvert\le C_\gb r^\gb,
\]
where
\[
C_\gb:= (2\pi)^{-1} \sum_{\ell=0}^\infty
A_\ell(2\ell+1) \bigl(\ell(\ell+1)\bigr)^{\gb/2}.
\]
This finishes the proof of the lemma.
\end{pf}
Lemma~\ref{lemHoeldercontkernelk} asserts
H\"older continuity of $k(r)$ at $r=0$
in terms of a $\ell^1$ criterion
on the angular power spectrum of
the isotropic GRF~$T$, while we provided $\ell^2$ criteria in
Section~\ref{secPowSpecDec}.
To relate these criteria, we first observe that for $\varepsilon> 0$
\[
\sum_{\ell=0}^\infty A_\ell
\ell^{1+\gb} \le\zeta(1+\varepsilon)^{1/2} \Biggl(\sum
_{\ell=0}^\infty A_\ell^2
\ell^{3+2\gb+\varepsilon} \Biggr)^{1/2}
\]
by the Cauchy--Schwarz inequality, where $\zeta$ denotes the
Riemann zeta function.
This implies with Theorem~\ref{thmBesovEquivNorm} that
Assumption~\ref{asssummability-Al} with $d=3$ is satisfied
if the kernel~$k_I$ is in $V^\eta(-1,1)$ for some $\eta> \gb+ 1$.

Our next step is to give bounds on moments of
$T(x) - T(y)$ for $x,y \in\IS^2$
in terms of the geodesic distance $d(x,y)$.
We prove the lemma by expressing the moments in
terms of the kernel~$k$
and by an application of the preceding lemma.

\begin{lemma}\label{lemboundmomentsTx-Ty}
Let $T$ be an isotropic Gaussian random field on~$\IS^2$
with angular power spectrum $(A_\ell, \ell\in\N_0)$.
If the angular power spectrum satisfies
Assumption~\ref{asssummability-Al} with $d=3$
for some $\gb\in[0,2]$,
then for all $p \in\N$ there exists a constant~$C_{\gb,p}$ such that
for all $x,y \in\IS^2$,
\[
\E\bigl(\bigl\llvert T(x) - T(y)\bigr\rrvert^{2p}\bigr) \le
C_{\gb,p} \,d(x,y)^{\gb p}.
\]
\end{lemma}

\begin{pf}
First note that $T(x) - T(y)$ is a centered Gaussian random variable.
Furthermore, if $X$ is a $\cN(0,\gs^2)$ distributed random variable,
then
\[
\E\bigl(\llvert X\rrvert^{2p}\bigr) = \E\bigl(\llvert\gs Y\rrvert
^{2p}\bigr) = \bigl(\gs^2\bigr)^p \E\bigl(
\llvert Y\rrvert^{2p}\bigr) = \E\bigl(X^2
\bigr)^p c_{2p}
\]
for $p \in\N$, where $Y$ is a standard normally distributed random
variable and $c_{2p}$ denotes the $2p$th moment of~$Y$.
We also observe that $\E(\llvert T(x) - T(y)\rrvert ^2)$ can be
expressed in terms
of~$k$ since
\begin{eqnarray*}
\E\bigl(\bigl\llvert T(x) - T(y)\bigr\rrvert^2\bigr) & = &\E
\bigl(T(x)^2\bigr) - 2 \E\bigl(T(x)T(y)\bigr) + \E
\bigl(T(y)^2\bigr)
\\
& = &k_T(x,x) - 2 k_T(x,y) + k_T(y,y)
\\
& = &2 \bigl( k(0) - k\bigl(d(x,y)\bigr) \bigr).
\end{eqnarray*}
Combining the two previous observations, we conclude that
\begin{eqnarray*}
\E\bigl(\bigl\llvert T(x) - T(y)\bigr\rrvert^{2p}\bigr) & = &c_{2p} \E\bigl(\bigl\llvert T(x) - T(y)\bigr\rrvert^2
\bigr)^p
\\
& = &2 c_{2p} \bigl(k(0) - k\bigl(d(x,y)\bigr)\bigr)^p
\\
& \le&2 c_{2p} C_\gb^p d(x,y)^{\gb p},
\end{eqnarray*}
where we applied Lemma~\ref{lemHoeldercontkernelk} in the last step.
This completes the proof of the lemma.
\end{pf}

The following result is a version of the Kolmogorov--Chentsov theorem for
random fields with domain~$\IS^2$, which is a special version of
Theorem~3.5 in~\citet{AL14} and proven here independently for
completeness. Note that in this result the fields do not have to be
Gaussian or isotropic.

\begin{theorem}[(Kolmogorov--Chentsov theorem)]\label{thmKC-sphere}
Let $T$ be a random field on the sphere that satisfies for some $p>0$
and some $\varepsilon\in(0,1]$
that there exists a constant~$C$ such that
\[
\E\bigl(\bigl\llvert T(x) - T(y)\bigr\rrvert^p\bigr) \le C
\,d(x,y)^{2 + \varepsilon p}
\]
for all $x,y \in\IS^2$. Then there exists a continuous modification of~$T$
that is locally H\"older continuous with exponent~$\gg$ for all $\gg
\in(0,\varepsilon)$.
\end{theorem}

\begin{pf}
Let us first construct six charts $(U_i, i=1,\ldots, 6)$ that cover
the sphere by taking the six possible hemispheres given by the
coordinate system such that the boundary is a circle of radius~$r$ with
$r \in(\sqrt{2/3},1)$; that is, we take a bit less than the complete
hemispheres but enough to cover the whole sphere. Let the coordinate
maps $(\varphi_i, i=1,\ldots,6)$ be given by the projection onto the
plane that divides the hemispheres; that is, if $U$ is contained in the
northern hemisphere, then the corresponding coordinate map~$\varphi$
is given by $\varphi((x_1,x_2,x_3)):= (x_1,x_2)$ for $x=
(x_1,x_2,x_3) \in U$ and maps onto the disc $\{x \in\R^2, \llVert
x\rrVert _{\R
^2} < r\}$.

For a given chart $(U,\varphi)$, we have to show that the Euclidean
norm in~$\R^2$ is equivalent to the metric on~$\IS^2$, that is, that
there exist constants $C_1, C_2 >0$ such that for all $x,y \in U$
\[
C_1 \bigl\llVert\varphi(x) - \varphi(y)\bigr\rrVert
_{\R^2} \le d(x,y) \le C_2 \bigl\llVert\varphi(x) -
\varphi(y)\bigr\rrVert_{\R^2}
\]
or equivalently that
\[
C_1 \le\frac{\arccos(\langle x,y\rangle_{\R^3})}{\llVert \varphi(x) -
\varphi(y)\rrVert _{\R^2}} \le C_2.
\]
We show this estimate for $U$ contained in the northern hemisphere. The
calculations for the other five charts are similar, and the bounds are
the same due to symmetry.

One first calculates that
\[
\langle x,y\rangle_{\R^3} = 1 - \tfrac{1}{2}\bigl(\bigl\llVert
\varphi(x) - \varphi(y)\bigr\rrVert_{\R^2}^2 + \llvert
x_3-y_3\rrvert^2\bigr)
\]
and shows that
\[
0 \le\llvert x_3-y_3\rrvert^2 \le
\frac{2r^2}{1-r^2} \bigl\llVert\varphi(x) - \varphi(y)\bigr\rrVert
_{\R^2}^2.
\]
This implies that we can bound the quotient of interest from above and
below by
\begin{eqnarray*}
&& \frac{\arccos(1 - ({1}/{2})\llVert \varphi(x) - \varphi(y)\rrVert
_{\R
^2}^2)}{\llVert \varphi(x) - \varphi(y)\rrVert _{\R^2}}
\\
&&\qquad  \le \frac
{\arccos(\langle x,y\rangle_{\R^3})}{\llVert \varphi(x) -
\varphi(y)\rrVert _{\R^2}}
\\
&&\qquad  \le\frac{\arccos(1 - ({1}/{2} + {r^2}/{(1-r^2)})\llVert \varphi
(x) - \varphi(y)\rrVert _{\R^2}^2)}{\llVert \varphi(x) - \varphi
(y)\rrVert _{\R^2}},
\end{eqnarray*}
since $\arccos$ is a monotonically decreasing function.
Let us define the function $f\dvtx[0,2r) \rightarrow\R$ by
\[
f(a):= \frac{\arccos(1 - \ga a^2)}{a}
\]
for $a \in(0,2r)$, where $\ga= 1/2, 1/2 + r^2/(1-r^2)$.
Then one shows with standard methods from real analysis that $f$ is
well defined on $[0,2r)$ and monotonically increasing,
which leads with the observation that $f(0) = \sqrt{2\ga}$ by l'H\^opital's
rule to the conclusion that
\[
C_1:= 1 \le\frac{\arccos(\langle x,y\rangle_{\R^3})}{\llVert \varphi
(x) -
\varphi(y)\rrVert _{\R^2}} \le\frac{\arccos((2r^4+3r^2-1)/(r^2-1))}{2r} =:
C_2 < + \infty
\]
and completes the proof of the equivalence of geodetic and Euclidean
distances on the sphere and in the charts.

For $a,b \in\varphi(U)$ it holds for the random field on the chart
by our assumptions and the equivalence of the distances that
\begin{eqnarray*}
\E\bigl(\bigl\llvert T\bigl(\varphi^{-1}(a)\bigr) - T\bigl(
\varphi^{-1}(b)\bigr)\bigr\rrvert^p\bigr) & \le& C \cdot d
\bigl(\varphi^{-1}(a),\varphi^{-1}(b)\bigr)^{2/p + \varepsilon}
\\
& \le& C \cdot C_2^{2/p + \varepsilon} \llVert a-b\rrVert^{2/p +
\varepsilon}.
\end{eqnarray*}
Since $\varphi(U)$ is a domain in~$\R^2$, we obtain by the
Kolmogorov--Chentsov theorem for domains [see Theorem~2.1
in~\citet{MS03}, Theorem~4.5 in~\citet{P092} or Theorem~3.1
in~\citet{AL14}] that there exists a continuous
modification~$T_1\circ\varphi^{-1}$ that is locally H\"older
continuous with exponent~$\gg$ for all $\gg\in(0,\varepsilon)$ and so
is $T_1$ on~$U$ due to the smoothness of the coordinate map.

With the same proof we obtain continuous modifications
$(T_i, i=1,\ldots,6)$ on all charts $(U_i, i=1,\ldots,6)$.
We glue these together with a smooth partition of unity $(\rho_i,
i=1,\ldots,6)$
on~$\IS^2$, which is subordinate to the open
covering [see, e.g., Theorem~1.73 in~\citet{L09}]
by
\[
\tilde{T}(x):= \sum_{i=1}^6
\rho_i(x) T_i(x)
\]
for all $x \in\IS^2$, where $T_i(x) = 0$ for $x \notin U_i$.
Then $\tilde{T}$ is a continuous modification of~$T$ that is
locally H\"older continuous with the same exponent~$\gg$ for all
$\gg\in(0,\varepsilon)$ due to the smoothness of the partition of unity.
This completes the proof of the theorem.
\end{pf}

With the made observations, we are now prepared to prove one of
the main results of this section which states that if the angular power spectrum
of an isotropic Gaussian random field is summable with weights $\ell
^{1+\gb}$,
then there exists a continuous modification which is H\"older continuous
with exponent~$\gg$ for all $\gg<\gb/2$.

\begin{theorem}\label{thmHoeldercontiGRF}
Let $T$ be an isotropic Gaussian random field on~$\IS^2$
with angular power spectrum $(A_\ell, \ell\in\N_0)$.
If the angular power spectrum satisfies
Assumption~\ref{asssummability-Al} with $d=3$
for some $\gb\in(0,2]$,
then there exists a continuous modification of~$T$
that is H\"older continuous with exponent~$\gg$
for all $\gg< \gb/2$.
\end{theorem}

\begin{pf}
The claim follows by the application of the previous results in the
following way:
it holds by Lemma~\ref{lemboundmomentsTx-Ty} that
for all $p \in\N$ and $x,y \in\IS^2$ the random field satisfies
\[
\E\bigl(\bigl\llvert T(x) - T(y)\bigr\rrvert^{2p}\bigr) \le
C_{\gb,p} \,d(x,y)^{\gb p} = C_{\gb,p} \,d(x,y)^{2 + (\gb/2 - 1/p)2p}.
\]
Theorem~\ref{thmKC-sphere} finally implies that there exists a
continuous modification that is locally H\"older continuous with
exponent~$\gg$ for all $\gg<\gb/2 - 1/p$ for any $p \in\N$, that
is, with exponent~$\gg$ for all $\gg< \gb/2$.
\end{pf}
Just as an example, let us calculate the parameters of
$\IP$-almost sure H\"older continuity for the two choices of~$\ga$
that we simulate in the following sections. Therefore let the angular
power spectrum of~$T$ be given by $A_\ell:= (\ell+1)^{-\ga}$ for
$\ell\in\N_0$. For $\ga= 3$ we get $\gb< 1$ which implies
$\gg<1/2$ in Theorem~\ref{thmHoeldercontiGRF} and $\ga=5$ implies
$\gb= 2$ and therefore $\gg<1$.

Furthermore as second main result of this section we are interested in the
assumptions on the angular power spectrum that imply the existence of
differentiable modifications
of isotropic GRFs.
In particular in the context of approximate, numerical solutions
of partial differential equations,
regularity properties of samples are essential for
the derivation of convergence rates for, for example, finite element or
finite difference
discretizations.

\begin{theorem}
\label{thmPathRegIsoGRF}
Let $T$ be a centered, isotropic Gaussian random field on~$\IS^2$ with
angular power spectrum
$(A_\ell, \ell\in\N_0)$.
If the angular power spectrum satisfies
Assumption~\ref{asssummability-Al} with $d=3$
%
for some $\beta> 0$,
then there exists a $C^\gg(\IS^2)$-valued modification of~$T$
for all $\gg< \gb/2$; that is, the modification is
$k$-times continuously differentiable with $k=\lceil\gb/2 \rceil-1$, and
the $k$th derivatives are H\"older continuous with exponent $\gg-k$.
\end{theorem}
\begin{pf}
Let us first observe that the made assumptions imply that
$T$ has a continuous modification by Theorem~\ref{thmHoeldercontiGRF}.
Without loss of generality let $T$ already be the continuous modification,
which is an isotropic Gaussian random field with the same parameters
and has
a Kar\-hunen--Lo\`eve expansion with the same parameters by
Corollary~\ref{corKLexpiGRF}.
Let $k:= \lceil\gb/2 \rceil-1$. Then
\begin{eqnarray*}
&& \E\bigl( \bigl\llVert(1-\Delta_{\IS^2})^{k/2} T\bigr\rrVert
_{L^2(\IS^2)}^2\bigr)
\\
&&\qquad  = \E\Biggl( \Biggl\llVert\sum
_{\ell=0}^\infty\sum_{m=-\ell}^\ell
a_{\ell m} \bigl(1+\ell(\ell+1)\bigr)^{k/2} Y_{\ell m}
\Biggr\rrVert_{L^2(\IS
^2)}^2 \Biggr)
\\
&&\qquad = \E\Biggl( \sum_{\ell=0}^\infty\sum
_{m=-\ell}^\ell a_{\ell
m}^2
\bigl(1+\ell(\ell+1)\bigr)^{k} \Biggr)
\\
&&\qquad  = \sum_{\ell=0}^\infty\Biggl(
A_\ell+ 2 \sum_{m=1}^\ell
A_\ell/2 \Biggr) \bigl( 1 + \bigl(\ell(\ell+1) \bigr)
\bigr)^k
\\
&&\qquad  = \sum_{\ell=0}^\infty A_\ell(
\ell+1) \bigl( 1 + \bigl(\ell(\ell+1) \bigr) \bigr)^k
\\
&&\qquad  < + \infty
\end{eqnarray*}
by the made assumptions. Furthermore $(1-\Delta_{\IS^2})^{k/2} T$ is
a continuous Gaussian random field by the properties of the Karhunen--Lo\`eve expansion with angular power spectrum $(A_\ell(\ell+1)
( 1 + (\ell(\ell+1) ) )^k, \ell\in\N_0)$.
Applying Theorem~\ref{thmHoeldercontiGRF} with parameter $\gb- 2k$,
we obtain that $(1-\Delta_{\IS^2})^{k/2} T$ is H\"older continuous
with exponent $\gg'$ for all $\gg' < \gb/2-k$, since it was already
continuous. By Theorem~XI.2.5 in~\citet{T81}, we obtain that $T
\in C^{k + \gg'}(\IS^2)$, where $m=-k$ in the framework of that
theorem. Setting $\gg:= k + \gg'$, we conclude that $T$ is in $C^{\gg
}(\IS^2)$ for all $\gg< \gb/2$ by the definition of~$\gg'$, which
completes the proof of the theorem.
\end{pf}
The results of this section can be extended to Gaussian random fields
on~$\IS^{d-1}$, $d\ge2$, with the methods introduced here and results
from~\citet{Y83}. The regularity result for H\"older continuity
and differentiability as generalizations of Theorems~\ref
{thmHoeldercontiGRF} and~\ref{thmPathRegIsoGRF} is stated in what
follows before we sketch how to adapt the proofs of the previous
results. The obtained results improve Theorem~II.2.11 in~\citet
{Y83} for H\"older continuity, since our version of the theorem does
not require the logarithmic term in the summability assumption on the
angular power spectrum.

\begin{theorem}
Let $T$ be a centered, isotropic Gaussian random field on~$\IS^{d-1}$
with angular power spectrum
$(A_\ell, \ell\in\N_0)$.
If\vspace*{2pt} the angular power spectrum satisfies
Assumption~\ref{asssummability-Al}
for some $\beta> 0$,
%
%
then there exists a $C^\gg(\IS^{d-1})$-valued modification of~$T$ for
all $\gg< \gb/2$; that is,
the modification is $k$-times continuously differentiable for $k=\lceil
\gb/2 \rceil-1$,
and the $k$th derivatives are H\"older continuous with exponent $\gg-k$.
\end{theorem}

\begin{pf}
Let us first consider $\gb\in(0,2]$. We obtain as generalization of
Lemma~\ref{lemHoeldercontkernelk} with Theorem~II.2.9 in~\citet
{Y83} for $r \in[0,\pi]$ that
\[
\bigl\llvert k(0) - k(r)\bigr\rrvert\le C_\gb r^\gb,
\]
where $C_\gb$ denotes a constant that does not depend on~$r$, and we
set $\gg(r) = r^\gb$ in that theorem. Lemma~\ref
{lemboundmomentsTx-Ty} extends in a one-to-one fashion to~$\IS^{d-1}$,
so that we conclude that
\[
\E\bigl(\bigl\llvert T(x) - T(y)\bigr\rrvert^{2p}\bigr) \le
C_{\gb,p} \,d(x,y)^{\gb p}
\]
for all $x,y \in\IS^{d-1}$ and $p \in\N$, where $C_{\gb,p}$
depends on the indicated parameters.
Using conformal charts on~$\IS^{d-1}$, the Kolmogorov--Chentsov
theorem on manifolds proven in Theorem~3.5 in~\citet{AL14}
implies that there exists a continuous modification of~$T$ that is H\"
older continuous with exponent~$\gg<\gb/2$. This completes the proof
of the theorem for $\gb\in(0,2]$.

For $\gb> 2$ we have to extend Theorem~\ref{thmPathRegIsoGRF} to
arbitrary dimensions~$d$. This can be done equivalently since
Theorem~XI.2.5 in~\citet{T81} holds for all compact manifolds. We
first recall that the Laplace--Beltrami operator~$\Delta_{\IS^{d-1}}$
on~$\IS^{d-1}$ has the spherical harmonics $(S_{\ell m}, \ell\in\N
_0, m=1,\ldots,h(\ell,d))$ as eigenbasis with eigenvalues given by
\[
\Delta_{\IS^{d-1}} S_{\ell m} = - \ell(\ell+d-2)S_{\ell m}
\]
for $\ell\in\N_0$ and $m=1,\ldots,h(\ell,d)$; see, for example,
\citet{AH12}, Section~3.3. Let us assume that $T$ is already
continuous without loss of generality by the first part of the proof.
Then the main calculation in the proof of Theorem~\ref
{thmPathRegIsoGRF} reads in the general case for $k:= \lceil\gb/2
\rceil-1$
\begin{eqnarray*}
&& \E\bigl( \bigl\llVert(1-\Delta_{\IS^{d-1}})^{k/2} T\bigr\rrVert
_{L^2(\IS
^{d-1})}^2 \bigr)
\\
&& \qquad= \E\Biggl( \Biggl\llVert\sum_{\ell=0}^\infty
\sum_{m=1}^{h(\ell,d)} a_{\ell m} \bigl(1+
\ell(\ell+d-2)\bigr)^{k/2} S_{\ell
m}\Biggr\rrVert
_{L^2(\IS^{d-1})}^2 \Biggr)
\\
&& \qquad= \E\Biggl( \sum_{\ell=0}^\infty\sum
_{m=1}^{h(\ell,d)} a_{\ell m}^2
\bigl(1+\ell(\ell+d-2)\bigr)^{k} \Biggr)
\\
&& \qquad= \sum_{\ell=0}^\infty
A_\ell h(\ell,d) \bigl( 1 + \bigl(\ell(\ell+d-2) \bigr)
\bigr)^k
\\
&& \qquad< + \infty,
\end{eqnarray*}
since $h(\ell,d) \simeq\ell^{d-2}$ by Stirling's inequalities and
$2k < \gb$.
The first part of the proof for $\gb\in(0,2]$ implies for the
continuous Gaussian random field $(1-\Delta_{\IS^{d-1}})^{k/2} T$,
which has the corresponding angular power spectrum $(A_\ell h(\ell,\break d)
( 1 + (\ell(\ell+d-2) ) )^k, \ell\in\N_0)$,
that it is H\"older continuous with exponent~$\gg'$ for all $\gg'<
\gb/2-k$.
Theorem~XI.2.5 in~\citet{T81}
again yields that $T \in C^{k + \gg'}(\IS^{d-1})$, and the proof of
the second part
of the theorem is finished in the same way as for~$\IS^2$.
\end{pf}
We remark that an alternative argument for the
H\"older sample regularity on~$\IS^2$ which avoids resorting
to H\"older regularity theory for elliptic pseudodifferential operators on
manifolds is sketched in \citet{Herrm2013}.
There, for even exponents~$k$, H\"older regularity was inferred from
Schauder estimates for (integer) powers of the Laplace--Beltrami operator,
and the result for general H\"older exponents was obtained by interpolation.
%
\section{Approximation of isotropic Gaussian random fields}\label{secapproxiGRF}
Let us approximate and simulate isotropic Gaussian random fields
in this section, where we use the properties of the random fields
that were introduced in Section~\ref{seciGRF}.
In what follows, we consider centered random fields without loss
of generality. It is clear by Corollary~\ref{corKLexpiGRF} that
we can transform the centered, isotropic random field into
a field with nonzero expectation by adding the expectation,
which is a constant according to Lemma~\ref{lempropalm}.
To prepare the presentation of the approximation of isotropic GRFs on~$\IS^2$,
we rewrite its series expansions, where we use the properties
of the spherical harmonic functions and the structure of real-valued
random fields.

\begin{lemma}\label{lemsimulatesequenceexpT}
Let $T$ be a centered, isotropic Gaussian random field.
For $\ell\in\N$, $m = 1,\ldots,\ell$, and $\vartheta\in[0,\pi
]$, set
\[
L_{\ell m}(\vartheta):= \sqrt{\frac{2\ell+1}{4\pi} \frac{(\ell
-m)!}{(\ell+m)!}}
P_{\ell m}(\cos\vartheta).
\]
Then for $y= (\sin\vartheta\cos\varphi, \sin\vartheta\sin\varphi
, \cos\vartheta)$ there holds
\begin{eqnarray*}
T(y) &=& \sum_{\ell=0}^\infty\Biggl(
\sqrt{A_\ell} X^1_{\ell0} L_{\ell
0}(
\vartheta)
\\
&&\hspace*{19pt}{} + \sqrt{2A_\ell} \sum_{m=1}^\ell
L_{\ell m}(\vartheta) \bigl(X^1_{\ell
m} \cos(m
\varphi) + X^2_{\ell m} \sin(m \varphi)\bigr)\Biggr)
\end{eqnarray*}
in law, where $((X^1_{\ell m}, X^2_{\ell m}), \ell\in\N_0,
m=0,\ldots, \ell)$ is a sequence of independent, real-valued,
standard normally distributed random variables and $X^2_{\ell0} = 0$
for \mbox{$\ell\in\N_0$}.
\end{lemma}

\begin{pf}
By Corollary~\ref{corKLexpiGRF} $T$ can be represented in the (mean
square convergent)
Kar\-hunen--Lo\`eve expansion
\[
T = \sum_{\ell=0}^\infty\sum
_{m=-\ell}^\ell a_{\ell m} Y_{\ell m}.
\]
This sum can be rewritten to
\begin{eqnarray*}
T & = &\sum_{\ell=0}^\infty
\Biggl(a_{\ell0} Y_{\ell0} + \sum_{m=1}^\ell(a_{\ell m}
Y_{\ell m} + a_{\ell-m} Y_{\ell-m}) \Biggr)
\\
& = &\sum_{\ell=0}^\infty\Biggl(a_{\ell0}
L_{\ell0}(\vartheta) + \sum_{m=1}^\ell
\bigl(a_{\ell m} Y_{\ell m} + (-1)^m
\overline{a_{\ell
m}} (-1)^m \overline{Y_{\ell m}}\bigr)
\Biggr)
\\
& = &\sum_{\ell=0}^\infty\Biggl(a_{\ell0}
L_{\ell0}(\vartheta) + \sum_{m=1}^\ell(a_{\ell m}
Y_{\ell m} + \overline{a_{\ell m} Y_{\ell m}})\Biggr)
\\
& = &\sum_{\ell=0}^\infty\Biggl(a_{\ell0}
L_{\ell0}(\vartheta) + \sum_{m=1}^\ell2
\Re(a_{\ell m} Y_{\ell m}) \Biggr)
\end{eqnarray*}
by Lemma~\ref{lempropalm} and the properties of the spherical harmonic
functions.
We observe that
\[
Y_{\ell m}(\vartheta,\varphi) = L_{\ell m}(\vartheta)
e^{im\varphi} = L_{\ell m}(\vartheta) \bigl(\cos(m\varphi) + i \sin(m
\varphi)\bigr)
\]
for $(\vartheta,\varphi) \in[0,\pi]\times[0,2\pi)$
and therefore by the properties of complex numbers that
\[
\Re\bigl(a_{\ell m} Y_{\ell m}(\vartheta,\varphi)\bigr) =
L_{\ell m}(\vartheta) \bigl(\Re a_{\ell m} \cos(m\varphi) - \Im
a_{\ell m} \sin(m\varphi) \bigr).
\]
Let $((X^1_{\ell m}, X^2_{\ell m}), \ell\in\N_0, m=0,\ldots, \ell
)$ be a
sequence of independent, real-valued, standard normally distributed
random variables,
then
\[
\Re a_{\ell m} = \sqrt{\frac{A_\ell}{2}} X^1_{\ell m}
\quad\mbox{and}\quad- \Im a_{\ell m} = \Im a_{\ell m} = \sqrt{
\frac{A_\ell}{2}} X^2_{\ell m}
\]
in law for $\ell\in\N$ and $m=1,\ldots,\ell$ by Corollary~\ref
{corKLexpiGRF}.
Furthermore the corollary implies that
\[
a_{\ell0} = \sqrt{A_\ell} X^1_{\ell0}
\]
for $\ell\in\N_0$.
The insertion of these observations
into the Kar\-hunen--Lo\`eve expansion of~$T$
completes the proof.
\end{pf}

For a given sequence $((X^1_{\ell m}, X^2_{\ell m}), \ell\in\N_0,
m=0,\ldots, \ell)$ as specified in Lemma~\ref
{lemsimulatesequenceexpT}, set
\begin{eqnarray*}
T(y) &:=& \sum_{\ell=0}^\infty\Biggl(\sqrt{A_\ell}
X^1_{\ell0} L_{\ell
0}(\vartheta)
\\[-3pt]
&&\hspace*{19pt}{} +
\sqrt{2A_\ell} \sum_{m=1}^\ell
L_{\ell m}(\vartheta) \bigl(X^1_{\ell
m} \cos(m
\varphi) + X^2_{\ell m} \sin(m \varphi)\bigr)\Biggr).
\end{eqnarray*}
In what follows,
we truncate the series expansion in order to implement it
and prove its convergence.
To this end for $\gk\in\N$ we set
\begin{eqnarray*}
T^\gk(y)
&:=& \sum_{\ell=0}^\gk\Biggl(\sqrt{A_\ell}
X^1_{\ell0} L_{\ell
0}(\vartheta)
\\[-3pt]
&&\hspace*{19pt}{}+
\sqrt{2A_\ell} \sum_{m=1}^\ell
L_{\ell m}(\vartheta) \bigl(X^1_{\ell
m} \cos(m
\varphi) + X^2_{\ell m} \sin(m \varphi)\bigr)\Biggr),
\end{eqnarray*}
where $y= (\sin\vartheta\cos\varphi, \sin\vartheta\sin\varphi,
\cos\vartheta)$ and $(\vartheta,\varphi) \in[0,\pi]\times[0,2\pi)$.

\begin{proposition}\label{propconvtruncKLexpGRF}
Let the angular power spectrum $(A_\ell, \ell\in\N_0)$ of the
centered, isotropic Gaussian random field~$T$
decay algebraically with order $\ga>2$; that is,
there exist constants $C>0$ and $\ell_0 \in\N$ such that
$A_\ell\le C \cdot\ell^{-\ga}$ for all $\ell> \ell_0$.
Then the series of approximate random fields $(T^\gk, \gk\in\N)$
converges to the random field~$T$ in~$L^2(\gO;L^2(\IS^2))$, and the
truncation error is bounded by
\[
\bigl\llVert T - T^\gk\bigr\rrVert_{L^2(\gO;L^2(\IS^2))} \le\hat{C}
\cdot
\gk^{-(\ga-2)/2}
\]
for $\gk\ge\ell_0$, where
\[
\hat{C}^2 = C \cdot\biggl(\frac{2}{\ga-2} + \frac{1}{\ga-1}
\biggr).
\]
\end{proposition}
\begin{pf}
Since $((X^1_{\ell m}, X^2_{\ell m}), \ell\in\N_0, m=0,\ldots, \ell)$
is a sequence of independent, standard normally distributed random variables,
the error is equal to
\begin{eqnarray*}
&& \bigl\llVert T - T^\gk\bigr\rrVert_{L^2(\gO;L^2(\IS^2))}
\\[-3pt]
&&\qquad = \sum_{\ell=\gk+1}^\infty\Biggl(
A_\ell\E\bigl(\bigl(X_{\ell0}^1
\bigr)^2\bigr) \llVert Y_{\ell0}\rrVert_{L^2(\IS^2)}^2
\\[-3pt]
&&\hspace*{64pt} {} + 2 A_\ell\sum_{m=1}^\ell \bigl( \E\bigl(\bigl(X_{\ell
m}^1\bigr)^2\bigr)
\llVert\Re Y_{\ell m}\rrVert_{L^2(\IS^2)}^2
\\[-3pt]
&&\hspace*{115pt}{}+ \E\bigl(
\bigl(X_{\ell m}^2\bigr)^2\bigr) \llVert\Im
Y_{\ell m}\rrVert_{L^2(\IS^2)}^2 \bigr) \Biggr)
\\[-3pt]
&&\qquad = \sum_{\ell=\gk+1}^\infty\Biggl(
A_\ell\llVert Y_{\ell0}\rrVert_{L^2(\IS^2)}^2
+ 2 A_\ell\sum_{m=1}^\ell
\bigl( \llVert\Re Y_{\ell m}\rrVert_{L^2(\IS^2)}^2 +
\llVert\Im Y_{\ell m}\rrVert_{L^2(\IS^2)}^2 \bigr)
\Biggr).
\end{eqnarray*}
We observe that $\llVert Y_{\ell0}\rrVert _{L^2(\IS^2)}^2 = 1$ and
$\llVert \Re Y_{\ell m}\rrVert _{L^2(\IS^2)}^2 + \llVert \Im Y_{\ell
m}\rrVert _{L^2(\IS
^2)}^2 = 1$
for $\ell\in\N_0$ and $m=1,\ldots, \ell$.
Therefore the sum simplifies to
\[
\bigl\llVert T - T^\gk\bigr\rrVert_{L^2(\gO;L^2(\IS^2))} = \sum
_{\ell=\gk+1}^\infty(2\ell+ 1) A_\ell,
\]
which is bounded by
\[
\sum_{\ell=\gk+1}^\infty(2\ell+ 1)
A_\ell\le C \sum_{\ell=\gk+1}^\infty
\bigl(2 \ell^{-(\ga-1)} + \ell^{-\ga}\bigr)
\]
due to the assumed properties of the angular power spectrum.
We rewrite the sum and bound it by the corresponding integral which
leads to
\begin{eqnarray*}
&& \sum_{\ell=\gk+1}^\infty\bigl(2
\ell^{-(\ga-1)} + \ell^{-\ga}\bigr)
\\
&&\qquad  = \sum
_{\ell=1}^\infty\bigl(2 (\ell+ \gk)^{-(\ga-1)} + (
\ell+ \gk)^{-\ga}\bigr)
\\
&&\qquad  \le\int_0^\infty\bigl(2 (x+
\gk)^{-(\ga-1)} + ( x + \gk)^{-\ga
} \bigr) \,dx
\\
&&\qquad  = \biggl(\frac{2}{\ga-2} + \frac{1}{\ga-1}\gk^{-1} \biggr) \gk
^{-(\ga-2)}.
\end{eqnarray*}
This completes the proof since $\gk^{-1}$ is bounded by~$1$.
\end{pf}

In an implementation in MATLAB we verified the theoretical results.
We took as ``exact'' solution the random fields with $\gk=2^7$ terms
(since for larger $\gk$ the elements of the angular power
spectrum~$A_\ell$,
and therefore the increments were so small that MATLAB failed to
compute the series expansion).
Instead of the $L^2(\IS^2)$ error in space, we used the maximum over
all grid points which is a stronger error. In Figure~\ref
{figGRFL2error} the results and the theoretical convergence rates are
shown for $\ga= 3,5$. One observes that the simulation results match
the theoretical results in Proposition~\ref{propconvtruncKLexpGRF}.

Since we discussed $\IP$-almost sure H\"older continuity in
Section~\ref{secHoeldercont},
we are also interested in $\IP$-almost sure convergence rates of
the approximate random fields $(T^\gk, \gk\in\N)$.
Therefore we include the following result on convergence in $L^p(\gO
;L^2(\IS^2))$
since we need it for optimal
pathwise convergence rates of the approximate random fields
$(T^\gk, \gk\in\N)$.

\begin{figure}

\includegraphics{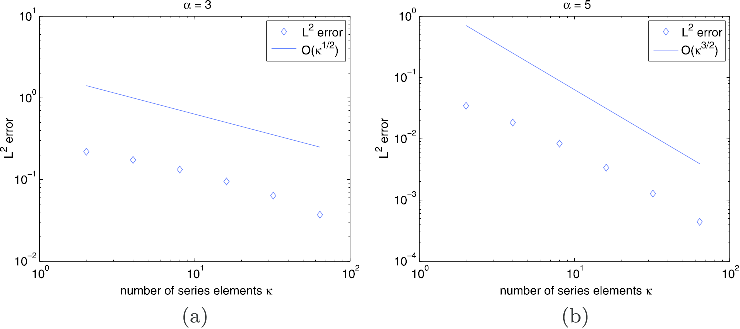}

\caption{Mean square error of the approximation of Gaussian random
fields with different
angular power spectrum and $1000$~Monte Carlo samples.
\textup{(a)}~Angular power spectrum with parameter $\ga=3$.
\textup{(b)}~Angular power spectrum with parameter $\ga=5$.}\label{figGRFL2error}
\end{figure}
%

\begin{theorem}\label{thmiGRFLpconv}
Let the angular power spectrum $(A_\ell, \ell\in\N_0)$ of the
centered, isotropic Gaussian random field~$T$
decay algebraically with order $\ga>2$; that is,
there exist constants $C>0$ and $\ell_0 \in\N$ such that $A_\ell\le
C \cdot\ell^{-\ga}$ for all $\ell> \ell_0$.
Then the series of approximate random fields $(T^\gk, \gk\in\N)$ converges
to the random field~$T$ in~$L^p(\gO;L^2(\IS^2))$ for any finite $p\ge1$,
and the truncation error is bounded by
\[
\bigl\llVert T - T^\gk\bigr\rrVert_{L^p(\gO;L^2(\IS^2))} \le
\hat{C}_p \cdot\gk^{-(\ga-2)/2}
\]
for $\gk\ge\ell_0$, where $\hat{C}_p$ is a constant that depends
on~$p$, $C$ and~$\ga$.
\end{theorem}

\begin{pf}
For $p \le2$ the result follows with Proposition~\ref
{propconvtruncKLexpGRF} and H\"older's inequality. Therefore let us
consider $p>2$ now. We prove the claim for $p=2m$, $m \in\N$. For all
other $p \in\R_+$, the result follows again by H\"older's inequality.
So let $m \in\N$; then Corollary~2.17 in~\citet{DPZ92} states
that there exists a constant~$C_m$ such that
\[
\bigl\llVert T - T^\gk\bigr\rrVert_{L^{2m}(\gO;L^2(\IS^2))}^{2m}
\le C_m \bigl\llVert T - T^\gk\bigr\rrVert
_{L^2(\gO;L^2(\IS^2))}^{2m}.
\]
Applying Proposition~\ref{propconvtruncKLexpGRF} we conclude that
\[
\bigl\llVert T - T^\gk\bigr\rrVert_{L^{2m}(\gO;L^2(\IS^2))}
\le(C_m)^{1/(2m)} \hat{C} \cdot\gk^{-(\ga-2)/2},
\]
where $\hat{C}$ is defined in Proposition~\ref{propconvtruncKLexpGRF},
%
which completes the proof.
\end{pf}

We have just shown that the convergence rate does not depend on~$p$.
This is necessary to get up to an epsilon the same sample convergence rates
as in the $p$th moment by the Borel--Cantelli lemma,
which we show in what follows.

\begin{corollary}\label{coriGRFP-asconv}
Let the angular power spectrum $(A_\ell, \ell\in\N_0)$ of the
centered, isotropic Gaussian random field~$T$
decay algebraically with order $\ga>2$; that is,
there exist constants $C>0$ and $\ell_0 \in\N$ such that
$A_\ell\le C \cdot\ell^{-\ga}$ for all $\ell> \ell_0$.
Then the series of approximate random fields $(T^\gk, \gk\in\N)$
converges to the random field~$T$ $\IP$-almost surely, and
for all $\gb< (\ga-2)/2$
the truncation error is asymptotically bounded by
\[
\bigl\llVert T - T^\gk\bigr\rrVert_{L^2(\IS^2)} \le
\gk^{-\gb}, \qquad\IP\mbox{-a.s.}
\]
\end{corollary}

\begin{pf}
Let $\gb< (\ga-2)/2$.
The Chebyshev inequality and Theorem~\ref{thmiGRFLpconv} imply that
\[
\IP\bigl(\bigl\llVert T-T^\gk\bigr\rrVert_{L^2(\IS^2)} \ge
\gk^{-\gb}\bigr) \le\gk^{\gb p} \E\bigl(\bigl\llVert
T-T^\gk\bigr\rrVert_{L^2(\IS^2)}^p\bigr) \le
\hat{C}_p^p \gk^{(\gb-(\ga-2)/2)p}.
\]
For all $p > ((\ga-2)/2-\gb)^{-1}$ the series
\[
\sum_{\gk=1}^\infty\gk^{(\gb-(\ga-2)/2)p} < +
\infty
\]
converges, and therefore the Borel--Cantelli lemma implies the claim.
\end{pf}
In Figure~\ref{figGRFpatherror}, we show the corresponding error plots
to Figure~\ref{figGRFL2error}, but instead of a Monte Carlo simulation
of the approximate $L^2(\gO;L^2(\IS^2))$ error, we plotted the error
of one sample. The convergence results coincide with the theoretical
results in Corollary~\ref{coriGRFP-asconv}.

\begin{figure}

\includegraphics{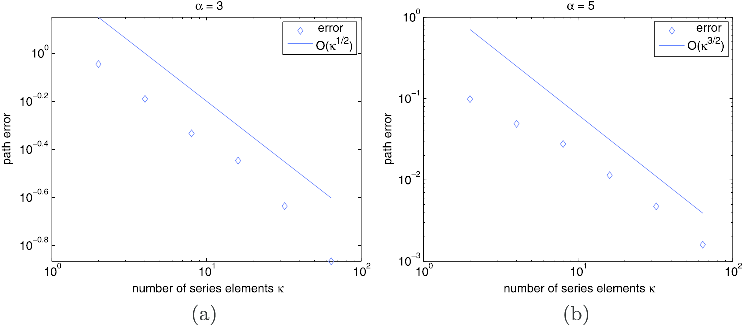}

\caption{Error of the approximation of a sample of Gaussian random
fields with different angular power spectrum.
\textup{(a)}~Angular power spectrum with parameter $\ga=3$.
\textup{(b)}~Angular power spectrum with parameter $\ga=5$.}\label{figGRFpatherror}
\end{figure}

%
\begin{figure}[b]

\includegraphics{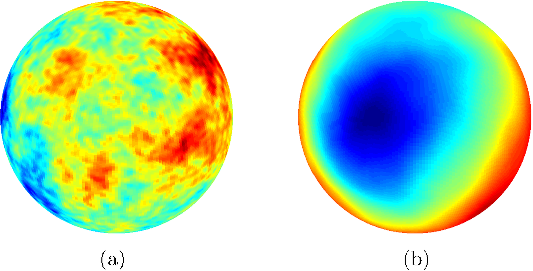}

\caption{Two samples of isotropic Gaussian random fields with
different angular power spectrum and truncation at $\gk=100$.
\textup{(a)}~Angular power spectrum with parameter $\ga=3$.
\textup{(b)}~Angular power spectrum with parameter $\ga=5$.}\label{figiGRF}
\end{figure}

To give the reader an idea of the structure of Gaussian random fields
in dependence of the decay of the angular power spectrum, we include
two samples in Figure~\ref{figiGRF}. Here we chose $A_\ell= (\ell
+1)^{-\ga}$ for $\ell\in\N_0$ and $\ga=3,5$. Therefore $A_\ell\le
\ell^{-\ga}$ for all $\ell\ge1$, which meets the assumptions of
Proposition~\ref{propconvtruncKLexpGRF}.
We plot the truncated series with $\gk= 100$ terms
(since larger $\gk$ do not affect the plots,
but the numerical accuracy suffers due to roundoff effects in MATLAB's
IEEE double precision format).
We remark that similarly to fast Fourier transforms,
there exist fast transforms for spherical harmonic functions [see,
e.g.,~\citet{M99}]
and the set of C routines \emph{SpharmonicKit} explained
in~\citet{HRKM03}.
These transforms
allow one to simulate isotropic Gaussian random fields with the
suggested approximations efficiently also for large choices of~$\gk$.

Analogously to the previous two sections, we finally want to give the
reader an idea of approximation results for isotropic Gaussian random
fields on spheres~$\IS^{d-1}$ in arbitrary dimensions $d\ge2$. So let
$T$ be an isotropic Gaussian random field on~$\IS^{d-1}$ for some
fixed $d \ge2$ with Kar\-hunen--Lo\`eve expansion
\[
T = \sum_{\ell=0}^\infty\sum
_{m=1}^{h(\ell,d)} a_{\ell m} S_{\ell m} =
\sum_{\ell=0}^\infty\sqrt{A_\ell}
\sum_{m=1}^{h(\ell,d)} X_{\ell m}
S_{\ell m},
\]
where $(X_{\ell m}, \ell\in\N_0, m=1,\ldots,h(\ell,d))$ is the
sequence of independent, standard normally distributed random variables
derived by $X_{\ell m} = a_{\ell m}/\sqrt{A_\ell}$.
We define similarly to $\IS^2$ the series of truncated random fields
$(T^\gk,\gk\in\N)$ by
\[
T^\gk
:= \sum_{\ell=0}^\gk
\sqrt{A_\ell} \sum_{m=1}^{h(\ell,d)}
X_{\ell
m} S_{\ell m}.
\]
Then we derive with the same computations as in the proofs of
Proposition~\ref{propconvtruncKLexpGRF}, Theorem~\ref{thmiGRFLpconv}
and Corollary~\ref{coriGRFP-asconv} convergence rates in $L^p$ and
$\IP$-almost sure sense that depend on the dimension~$d-1$ of the sphere.

\begin{theorem}
Let $T$ be a Gaussian isotropic random field on~$\IS^{d-1}$ with
angular power spectrum $(A_\ell, \ell\in\N_0)$ that
decays algebraically with order $\ga>2$; that is,
there exist constants $C>0$ and $\ell_0 \in\N$ such that $A_\ell\le
C \cdot\ell^{-\ga}$ for all $\ell> \ell_0$.
Then the series of approximate random fields $(T^\gk, \gk\in\N)$ converges
to the random field~$T$ in~$L^p(\gO;L^2(\IS^{d-1}))$ for any finite
$p\ge1$,
and the truncation error is bounded by
\[
\bigl\llVert T - T^\gk\bigr\rrVert_{L^p(\gO;L^2(\IS^{d-1}))} \le
C_p \cdot\gk^{-(\ga+1-d)/2}
\]
for $\gk\ge\ell_0$, where
$C_p>0$ is a constant that depends on~$d$, $p$ and~$\ga$.

Furthermore the series of approximate random fields
$(T^\gk, \gk\in\N)$ converges to the random field~$T$
$\IP$-almost surely, and
for all $\gb< (\ga+1-d)/2$
the truncation error is asymptotically bounded by
\[
\bigl\llVert T - T^\gk\bigr\rrVert_{L^2(\IS^{d-1})} \le
\gk^{-\gb}, \qquad\IP\mbox{-a.s.}
\]
%
\end{theorem}
\begin{pf}
This theorem is the generalization of Proposition~\ref
{propconvtruncKLexpGRF}, Theorem~\ref{thmiGRFLpconv} and
Corollary~\ref{coriGRFP-asconv}. The proofs of Theorem~\ref
{thmiGRFLpconv} and Corollary~\ref{coriGRFP-asconv} are exactly the
same except that the input parameters change. So it remains to show the
first claim of the theorem for $p=2$, which is the equivalent of
Proposition~\ref{propconvtruncKLexpGRF}. With the observation that the
spherical harmonics have norm one in $L^2(\IS^{d-1})$, the
independence of the normal random variables and that
$h(\ell,d) \simeq\ell^{d-2}$
by Stirling's inequalities,
we obtain
\[
\bigl\llVert T - T^\gk\bigr\rrVert_{L^2(\gO;L^2(\IS^{d-1}))} \le C \sum
_{\ell=\gk+1}^\infty\ell^{-(\ga- d + 2)}.
\]
The continuation of the proof of Proposition~\ref{propconvtruncKLexpGRF}
with these new parameters yields the claimed convergence rate.
\end{pf}
At this point we remark that a reference which is also devoted to
approximations of Gaussian isotropic random fields
on~$\IS^{d-1}$ is~\citet{KK01},
where the authors investigate different
types of errors and convergence than we do.
In this work
probabilities are bounded for $L^p(\IS^{d-1})$ estimates in space;
that is, quantities of the form
\[
\IP\bigl( \bigl\llVert T - T^\gk\bigr\rrVert_{L^p(\IS^{d-1})} >
\varepsilon\bigr) < \delta
\]
are considered.
These estimates cannot be used to derive neither convergence rates in
$L^p(\gO;L^2(\IS^{d-1}))$ nor $\IP$-almost sure convergence rates in
$L^2(\IS^{d-1})$ to the best of our knowledge.
Since the obtained bounds for the probabilities in Theorem~2
in~\citet{KK01}
do not depend on the truncation parameter~$\gk$, it
is not clear how the Borel--Cantelli lemma could be applied.
%
\section{Lognormal isotropic Gaussian random fields}\label{sechailstones}
In this section we consider lognormal random fields on~$\IS^2$; that
is, if $T$ is an isotropic Gaussian random field on~$\IS^2$, then we
are interested in $\exp(T)$ given by $\exp(T(x))$ for all $x \in\IS
^2$. These random fields are especially of interest when modeling
Saharan dust particles [see, e.g., \citet{NMR03}], feldspar
particles [cp.~\citet{VNKvdZ06}] and ice crystals
[cp.~\citet{NMF04}].
We show in the following that the sample regularity of a lognormal
random field is the same as that of the underlying Gaussian random
field. This is done by first proving regularity in $L^p(\gO;\R)$ and
then applying the Kolmogorov--Chentsov theorem similarly to
Section~\ref{secHoeldercont}.

\begin{lemma}\label{lemlognormLp-distance}
Let $T$ be an isotropic Gaussian random field on~$\IS^2$ with angular power
spectrum $(A_\ell, \ell\in\N_0)$.
If the angular power spectrum satisfies that $A_\ell\le C \ell^{-\ga
}$ for all $\ell\in\N$,
some $\ga> 2$ and some constant~$C$, then for all $p \in\N$ and $\gb
< \ga-2$, $\gb\le2$
there exists a constant~$C_{\gb,p}$ such that for all $x,y \in\IS^2$
it holds that
\[
\bigl\llVert\exp\bigl(T(x)\bigr) - \exp\bigl(T(y)\bigr)\bigr\rrVert
_{L^{p}(\gO;\R)} \le2 \exp\bigl(p k(0)\bigr) C_{\gb,p}
\,d(x,y)^{\gb/2}.
\]
\end{lemma}

\begin{pf}
Let us first observe that for $a,b \in\R$ it holds that
\[
\bigl\llvert e^a - e^b\bigr\rrvert= \biggl\llvert
\int_b^a e^z \,dz \biggr\rrvert\le
\llvert a-b\rrvert\max\bigl\{ e^a,e^b\bigr\} \le\llvert
a-b\rrvert\bigl(e^a + e^b\bigr).
\]
This implies for $x,y \in\IS^2$ that
\begin{eqnarray*}
&& \bigl\llVert\exp\bigl(T(x)\bigr) - \exp\bigl(T(y)\bigr)\bigr\rrVert
_{L^{p}(\gO;\R)}^p
\\
&& \qquad\le\E\bigl( \bigl(\exp\bigl(T(x)\bigr) + \exp\bigl(T(y)\bigr)
\bigr)^p \bigl\llvert T(x) - T(y)\bigr\rrvert^p \bigr)
\\
&& \qquad\le\E\bigl( \bigl(\exp\bigl(T(x)\bigr) + \exp\bigl(T(y)\bigr)
\bigr)^{2p} \bigr)^{1/2} \cdot\E\bigl(\bigl\llvert T(x) - T(y)
\bigr\rrvert^{2p}\bigr)^{1/2},
\end{eqnarray*}
where we applied H\"older's inequality in the last step.
By Lemma~\ref{lemboundmomentsTx-Ty} the second term is bounded by
\[
\E\bigl(\bigl\llvert T(x) - T(y)\bigr\rrvert^{2p}
\bigr)^{1/2} \le C_{\gb,p}^p d(x,y)^{p\gb/2}
\]
for any $\gb< \ga-2$, $\gb\le2$. The first term satisfies that
\begin{eqnarray*}
&& \E\bigl( \bigl(\exp\bigl(T(x)\bigr) + \exp\bigl(T(y)\bigr)
\bigr)^{2p} \bigr)^{1/2}
\\
&& \qquad\le2^{(2p-1)/2} \bigl(\E\bigl( \exp\bigl(2pT(x)\bigr)\bigr) + \E
\bigl( \exp\bigl(2pT(y)\bigr)\bigr) \bigr)^{1/2}.
\end{eqnarray*}
Since $T(x)$ and $T(y)$ are real-valued Gaussian random variables with
expectation zero and variance $k(0)$, the moment generating function is
given by
\[
\E\bigl( \exp\bigl(2pT(x)\bigr)\bigr) = \exp\bigl(2p^2 k(0)\bigr),
\]
which implies that
\begin{eqnarray*}
\E\bigl( \bigl(\exp\bigl(T(x)\bigr) + \exp\bigl(T(y)\bigr)
\bigr)^{2p} \bigr)^{1/2} & \le&2^{(2p-1)/2} 2^{1/2}
\exp\bigl(p^2 k(0)\bigr)
\\
& = &2^p \exp\bigl(p^2 k(0)\bigr).
\end{eqnarray*}
Therefore we conclude that
\begin{eqnarray*}
\bigl\llVert\exp\bigl(T(x)\bigr) - \exp\bigl(T(y)\bigr)\bigr\rrVert
_{L^{p}(\gO;\R)}
\le2 \exp\bigl(p k(0)\bigr) C_{\gb,p}
\,d(x,y)^{\gb/2},
\end{eqnarray*}
which completes the proof.
\end{pf}
The lemma enables us to conclude that the lognormal random field of an
isotropic Gaussian random field~$T$ has the same sample H\"older
continuity properties as~$T$.

\begin{corollary}\label{corHoelderContLogNormalGRF}
Let $T$ be an isotropic Gaussian random field on~$\IS^2$ with angular
power spectrum
$(A_\ell, \ell\in\N_0)$. If the angular power spectrum satisfies that
$A_\ell\le C \cdot\ell^{-\ga}$ for all $\ell\in\N$, some $\ga> 2$
and some constant~$C$, then there exists a modification of~$\exp(T)$ that
is H\"older continuous with exponent~$\gg$ for all $\gg< (\ga-2)/2$,
$\gg\le1$.
\end{corollary}

\begin{pf}
The proof is the same as the one of Theorem~\ref{thmHoeldercontiGRF},
where we apply Lemma~\ref{lemlognormLp-distance} instead of Lemma~\ref
{lemboundmomentsTx-Ty}.
\end{pf}

In Figure~\ref{fighailstones} we took the Gaussian random field
samples that are shown in Figure~\ref{figiGRF} and plotted the
deformed sphere with the corresponding lognormal radius which is done
when modeling dust or feldspar particles, respectively, ice crystals.

In Theorem~\ref{thmPathRegIsoGRF} we have shown the existence of
$k$-times continuously differentiable modifications of isotropic GRFs
depending on the convergence of the corresponding angular power spectrum.
The compactness of the unit sphere, the smoothness of the exponential function
and the chain rule imply as a direct consequence that
the same properties hold for the corresponding lognormal random fields.
%
\begin{corollary}
\label{corPathRegIsoLogNormRF}
Let $T$ be a centered, isotropic Gaussian random field on~$\IS^2$ with
angular power spectrum
$(A_\ell, \ell\in\N_0)$.
If the angular power spectrum satisfies
Assumption~\ref{asssummability-Al} with $d=3$
for some $\beta> 0$,
then there exists a $C^\gg(\IS^2)$-valued modification of the
corresponding lognormal random
field~$\exp(T)$ for all $\gg< \gb/2$; that is,
the modification is $k$-times continuously differentiable with
$k=\lceil\gb/2 \rceil-1$ and
the $k$th derivatives are H\"older continuous with exponent $\gg-k$.
\end{corollary}

\begin{figure}

\includegraphics{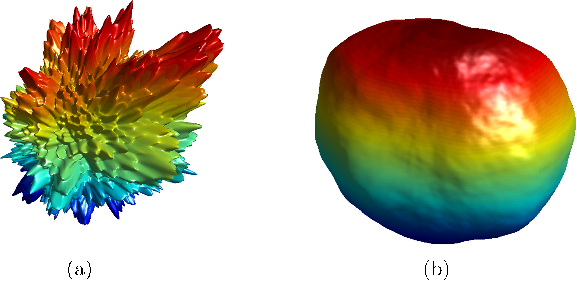}

\caption{Corresponding lognormal samples to Figure~\protect\ref
{figiGRF} with $\gk=100$.
\textup{(a)}~Angular power spectrum with parameter $\ga=3$.
\textup{(b)}~Angular power spectrum with parameter $\ga=5$.}\label{fighailstones}
\end{figure}
%

\begin{remark}
The results of this section directly
generalize to (isotropic) lognormal random fields on~$\IS^{d-1}$,
$d \ge2$, where the H\"older exponent obtained from the
decay condition $A_\ell\le C \cdot\ell^{-\ga}$
changes to $\gg< \ga- d+1$ in
Corollary~\ref{corHoelderContLogNormalGRF},
and Assumption~\ref{asssummability-Al} with $d=3$
in Corollary~\ref{corPathRegIsoLogNormRF} has to be replaced by
Assumption~\ref{asssummability-Al}.
%
%
\end{remark}
%
\section{Stochastic partial differential equations on the sphere}\label{secstochheateqn}
In this section we consider the heat equation on the sphere with
additive $Q$-Wiener noise as an example of a
stochastic partial differential equation (SPDE) on~$\IS^2$.
To discuss stochastic partial differential equations
we first introduce $Q$-Wiener processes on the sphere.

To this end
let us consider $Q$-Wiener processes that take
values in~$L^2(\IS^2)$ and that are isotropic in space.
Then, by Lemma~\ref{lemsimulatesequenceexpT} and by
the construction of \mbox{$Q$-}Wiener processes out of GRFs as was done
in an abstract setting, for example, in~\citet{DPZ92,PR07},
a $Q$-Wiener process taking values in~$L^2(\IS^2)$ can be
characterized by the
Kar\-hunen--Lo\`eve expansion
\begin{eqnarray*}
W(t,y) & = &\sum_{\ell=0}^\infty\sum
_{m=-\ell}^\ell a_{\ell m}(t)
Y_{\ell m}(y)
\\
& = &\sum_{\ell=0}^\infty\Biggl(
\sqrt{A_\ell} \gb^1_{\ell0}(t)
Y_{\ell0}(y)
\\
&&\hspace*{20pt} {}+ \sqrt{2A_\ell} \sum_{m=1}^\ell
\bigl(\gb^1_{\ell
m}(t) \Re Y_{\ell m}(y) +
\gb^2_{\ell m}(t) \Im Y_{\ell m}(y)\bigr) \Biggr)
\\
& = &\sum_{\ell=0}^\infty\Biggl(
\sqrt{A_\ell} \gb^1_{\ell0}(t) L_{\ell0}(
\vartheta)
\\
&&\hspace*{20pt} {} + \sqrt{2A_\ell} \sum_{m=1}^\ell
L_{\ell
m}(\vartheta) \bigl(\gb^1_{\ell m}(t) \cos(m
\varphi) + \gb^2_{\ell
m}(t) \sin(m \varphi)\bigr) \Biggr),
\end{eqnarray*}
where again $y= (\sin\vartheta\cos\varphi, \sin\vartheta\sin
\varphi, \cos\vartheta)$ and
$((\gb^1_{\ell m}, \gb^2_{\ell m}), \ell\in\N_0, m=0,\ldots, \ell)$
is a sequence of independent, real-valued Brownian motions with
$\gb^2_{\ell0} = 0$ for $\ell\in\N_0$ and $t \in\R_+$.
The covariance operator~$Q$ is characterized similarly to the
introduction in~\citet{LLS13} by
\begin{eqnarray*}
Q Y_{LM} & = &\sum_{\ell=0}^\infty
\sum_{m=-\ell}^\ell\E\bigl(
\bigl(W(1),Y_{LM}\bigr)_H \overline{\bigl(W(1),Y_{\ell m}
\bigr)_H} \bigr) Y_{\ell m}
\\
& = &\sum_{\ell=0}^\infty\sum
_{m=-\ell}^\ell\E\bigl( a_{LM}(1)
\overline{a_{\ell m}(1)} \bigr) Y_{\ell m}
\\
& = &A_L Y_{LM}
\end{eqnarray*}
for $L \in\N_0$ and $M=-L,\ldots,L$; that is, the eigenvalues of~$Q$
are characterized by the angular power spectrum $(A_\ell, \ell\in\N_0)$,
and the eigenfunctions are the spherical harmonic functions.

Let us calculate $\llVert W(t)\rrVert _{L^2(\gO;L^2(\IS^2))}$ for $t
\in\R_+$ next.
It holds similarly to the proof of Proposition~\ref{propconvtruncKLexpGRF} that
\begin{eqnarray*}
&& \bigl\llVert W(t)\bigr\rrVert_{L^2(\gO;L^2(\IS^2))}^2
\\
&&\qquad  = \sum_{\ell=0}^\infty\Biggl(
A_\ell\E\bigl(\bigl(\gb_{\ell0}^1(t)
\bigr)^2\bigr) \llVert Y_{\ell0}\rrVert_{L^2(\IS^2)}^2
\\
&&\hspace*{20pt}\quad\qquad{} + 2 A_\ell\sum_{m=1}^\ell
\bigl( \E\bigl(\bigl(\gb_{\ell
m}^1(t)\bigr)^2
\bigr) \llVert\Re Y_{\ell m}\rrVert_{L^2(\IS^2)}^2
\\
&&\hspace*{105pt}{} + \E
\bigl(\bigl(\gb_{\ell m}^2(t)\bigr)^2\bigr) \llVert
\Im Y_{\ell m}\rrVert_{L^2(\IS^2)}^2 \bigr) \Biggr)
\\
&&\qquad = t \sum_{\ell= 0}^\infty(2
\ell+1)A_\ell= t \trace Q.
\end{eqnarray*}
This expression is finite for any finite $t \in\R_+$, if
$\sum_{\ell=0}^\infty\ell A_\ell$ is finite.

With these definitions of
$Q$-Wiener processes as well as the Laplace operator on the sphere in
Section~\ref{seciGRF},
we are now in position
to write down the stochastic heat equation
%
\begin{equation}
\label{eqstochheateqn} dX(t) = \Delta_{\IS^2} X(t) \,dt + dW(t)
\end{equation}
with initial condition $X(0) = X_0 \in L^2(\gO;L^2(\IS^2))$, where $t
\in\IT= [0,T]$,\break \mbox{$T < + \infty$}.

Looking for solutions in~$L^2(\IS^2)$, we rewrite equation~(\ref
{eqstochheateqn}) to
\begin{eqnarray*}
X(t) &=& X_0 + \int_0^t
\Delta_{\IS^2} X(s) \,ds + \int_0^t
dW(s)
\\
&=& X_0 + \int_0^t
\Delta_{\IS^2} X(s) \,ds + W(t),
\end{eqnarray*}
and further, since the spherical harmonic functions~$\cY$ form an
orthonormal basis of~$L^2(\IS^2)$ and are eigenfunctions of~$\Delta
_{\IS^2}$, we have that
\begin{eqnarray*}
&& \sum_{\ell=0}^\infty\sum
_{m=-\ell}^\ell\bigl(X(t), Y_{\ell
m}\bigr)_{L^2(\IS^2)} Y_{\ell m}
\\[-3pt]
&& \qquad= \sum_{\ell=0}^\infty\sum
_{m=-\ell}^\ell\biggl( (X_0,
Y_{\ell m})_{L^2(\IS^2)} Y_{\ell m}
\\[-3pt]
&&\hspace*{46pt}\quad\qquad {}+ \int_0^t \bigl(X(s),
Y_{\ell m}\bigr)_{L^2(\IS^2)} \Delta_{\IS^2} Y_{\ell m}
\,ds + a_{\ell m}(t) Y_{\ell m} \biggr)
\\[-3pt]
&& \qquad= \sum_{\ell=0}^\infty\sum
_{m=-\ell}^\ell\biggl( (X_0,
Y_{\ell m})_{L^2(\IS^2)}
\\[-3pt]
&&\hspace*{45pt}\quad\qquad {}- \ell(\ell+1)\int_0^t
\bigl(X(s), Y_{\ell m}\bigr)_{L^2(\IS
^2)} \,ds + a_{\ell m}(t)
\biggr) Y_{\ell m}.
\end{eqnarray*}
This is equivalent to solve for all $\ell\in\N_0$ and $m=-\ell
,\ldots,\ell$ the stochastic (ordinary) differential equation
\begin{eqnarray*}
&&\bigl(X(t), Y_{\ell m}\bigr)_{L^2(\IS^2)}
\\[-3pt]
&&\qquad = (X_0,
Y_{\ell m})_{L^2(\IS^2)} - \ell(\ell+1) \int_0^t
\bigl(X(s), Y_{\ell m}\bigr)_{L^2(\IS^2)} \,ds + a_{\ell m}(t).
\end{eqnarray*}
The variations of constants formula yields
\[
\bigl(X(t), Y_{\ell m}\bigr)_{L^2(\IS^2)} = e^{-\ell(\ell+1)t}
(X_0, Y_{\ell m})_{L^2(\IS^2)} + \int_0^t
e^{-\ell(\ell+1)(t-s)} \,d a_{\ell m}(s).
\]
So overall the solution of stochastic heat equation~(\ref
{eqstochheateqn}) reads
\begin{eqnarray*}
X(t) & = &\sum_{\ell=0}^\infty\sum
_{m=-\ell}^\ell\biggl( e^{-\ell(\ell+1)t}
(X_0, Y_{\ell m})_{L^2(\IS^2)} + \int_0^t
e^{-\ell(\ell+1)(t-s)} \,d a_{\ell m}(s) \biggr) Y_{\ell
m}
\\[-3pt]
& = &\sum_{\ell=0}^\infty\Biggl( \sum
_{m=-\ell}^\ell e^{-\ell(\ell
+1)t} (X_0,
Y_{\ell m})_{L^2(\IS^2)} Y_{\ell m}
\\[-3pt]
&&\hspace*{20pt}{} + \sqrt{A_\ell} \Biggl( \int_0^t
e^{-\ell(\ell+1)(t-s)} \,d\gb^1_{\ell0}(s) Y_{\ell0}
\\[-3pt]
&&\hspace*{58pt} {} + \sqrt{2} \sum_{m=1}^\ell
\biggl( \int_0^t e^{-\ell(\ell+1)(t-s)} \,d
\gb^1_{\ell m}(s) \Re Y_{\ell
m}
\\[-3pt]
&&\hspace*{109pt} {} + \int_0^t
e^{-\ell(\ell+1)(t-s)} \,d \gb^2_{\ell m}(s) \Im Y_{\ell m}
\biggr) \Biggr) \Biggr)
\\[-4pt]
& =:& \sum_{\ell=0}^\infty X_\ell(t),
\end{eqnarray*}
and we choose the sequence of stochastic processes $(X_\ell, \ell\in
\N_0)$ accordingly.
Each process in this sequence satisfies the recursion formula
\begin{eqnarray*}
X_\ell(t+h) & = &e^{-\ell(\ell+1)h}X_\ell(t)
\\
&& {} + \sqrt{A_\ell} \Biggl( \int_t^{t+h}
e^{-\ell(\ell+1)(t+h-s)} \,d\gb^1_{\ell0}(s) Y_{\ell0}
\\
&&\hspace*{39pt} {}+ \sqrt{2} \sum_{m=1}^\ell
\biggl( \int_t^{t+h} e^{-\ell(\ell+1)(t+h-s)} \,d
\gb^1_{\ell m}(s) \Re Y_{\ell m}
\\
&&\hspace*{91pt} {} + \int_t^{t+h}
e^{-\ell(\ell+1)(t+h-s)} \,d \gb^2_{\ell m}(s) \Im Y_{\ell m}
\biggr) \Biggr).
\end{eqnarray*}
Similarly to~\citet{JK09}, we observe that by the It\^o formula
[see, e.g.,~\citet{K06}],
%
\begin{equation}
\label{eqstochconv} \int_0^t e^{-\ell(\ell+1)(t-s)} \,d
\gb^i_{\ell m} (s)
\end{equation}
is normally distributed with mean zero and
variance $(2\ell(\ell+1))^{-1} (1-e^{-2\ell(\ell+1)t})$ for
$\ell\in\N$, $m=1,\ldots,\ell$ and $i=1,2$.
This implies that
\[
\int_t^{t+h} e^{-\ell(\ell+1)(t+h-s)} \,d
\gb^i_{\ell m} (s) \sim\cN\bigl(0,\gs_{\ell h}^2
\bigr),
\]
where
\[
\gs_{\ell h}^2:= \bigl(2\ell(\ell+1)\bigr)^{-1}
\bigl(1-e^{-2\ell(\ell+1)h}\bigr).
\]
For $\ell= 0$ we have no convolution integral, and
therefore the distribution of the expression is that of (the increment of)
a standard Brownian motion, that is, \mbox{$\gs_{0 h}^2 = h$.}

For the simulation of paths of the solution, we have to compute the
solution on a discrete time grid $0 = t_0 < t_1 < \cdots< t_n = T$, $n
\in\N$, on which the path of the Brownian motion, respectively,
stochastic integral~(\ref{eqstochconv}), is known. Stochastic
integral~(\ref{eqstochconv}) at time~$t_k$, $k = 0,\ldots,n$, is
equal in law to a sum of weighted, standard normally distributed random
variables
\begin{eqnarray*}
&& \int_0^{t_k} e^{-\ell(\ell+1)(t_k-s)} \,d
\gb^i_{\ell m} (s)
\\
&& \qquad= \sum_{j=0}^{k-1} \int
_{t_j}^{t_{j+1}} e^{-\ell(\ell
+1)(t_k-s)} \,d
\gb^i_{\ell m} (s)
\\
&& \qquad= \sum_{j=0}^{k-1}
e^{-\ell(\ell+1)(t_k-t_{j+1})}\int_{t_j}^{t_{j+1}} e^{-\ell(\ell+1)(t_{j+1}-s)}
\,d\gb^i_{\ell m} (s)
\\
&& \qquad= \sum_{j=0}^{k-1}
e^{-\ell(\ell+1)\sum_{i=j+1}^k h_i} \gs_{\ell h_j} X^i_{\ell m}(j),
\end{eqnarray*}
where $h_j = t_{j+1} - t_j$, $j = 0,\ldots, n-1$ and $(X^i_{\ell
m}(j), \ell\in\N_0, m=0,\ldots,\ell, i=1,2, j=0,\ldots,n-1)$ is a
sequence of independent, standard normally distributed random
variables. This enables us to write down the solution of equation~(\ref
{eqstochheateqn}) recursively,
\begin{eqnarray*}
&& X_\ell(t_{k+1})
\\
&& \qquad= e^{-\ell(\ell+1)h_k} X_\ell(t_k)
\\
&&\quad\qquad {} + \sqrt{A_\ell} \gs_{\ell h_k} \Biggl(
X^1_{\ell0}(k) Y_{\ell0} + \sqrt{2} \sum
_{m=1}^\ell\bigl( X^1_{\ell m}(k)
\Re Y_{\ell m} + X^2_{\ell m}(k) \Im Y_{\ell m}
\bigr) \Biggr)
\end{eqnarray*}
for all $\ell\in\N_0$ and $k=0,\ldots,n-1$.
Using the notation of Lemma~\ref{lemsimulatesequenceexpT}, we rewrite
the recursion to
\begin{eqnarray*}
&& X_\ell(t_{k+1})
\\
&& \qquad= e^{-\ell(\ell+1)h_k} X_\ell(t_k) +
\psi_\ell(k)
\\
&& \qquad= e^{-\ell(\ell+1)t_{k+1}} \sum_{m=-\ell}^\ell(X_0,
Y_{\ell m})_{L^2(\IS^2)} Y_{\ell m} + \sum
_{j=0}^k e^{-\ell(\ell+1)\sum_{i=j+1}^{k} h_i} \psi_\ell(j),
\end{eqnarray*}
where the increments are given by
\begin{eqnarray*}
\psi_\ell(j,y) &=& \sqrt{A_\ell} \gs_{\ell h_j}
\Biggl( X^1_{\ell0}(j) L_{\ell0}(\vartheta)
\\
&&\hspace*{45pt}{} + \sqrt{2} \sum_{m=1}^\ell
L_{\ell m}(\vartheta) \bigl( X^1_{\ell m}(j) \cos(m
\varphi) + X^2_{\ell m}(j) \sin(m\varphi) \bigr) \Biggr)
\end{eqnarray*}
for $y= (\sin\vartheta\cos\varphi, \sin\vartheta\sin\varphi,
\cos\vartheta) \in\IS^2$ and $j=0,\ldots,n-1$.
We observe for later use that
\begin{eqnarray*}
&& \sum_{j=0}^k e^{-\ell(\ell+1) \sum_{i=j+1}^{k-1} h_i}
\psi_\ell(j)
\\
&& \qquad= \sqrt{A_\ell} \Biggl( \Biggl(\sum
_{j=0}^k e^{-\ell(\ell+1)\sum_{i=j+1}^{k} h_i} \gs_{\ell h_j}
X^1_{\ell0}(j)\Biggr) L_{\ell0}(\vartheta)
\\
&&\hspace*{27pt}\quad\qquad {}+ \sqrt{2} \sum_{m=1}^\ell
L_{\ell m}(\vartheta) \Biggl( \Biggl(\sum_{j=0}^k
e^{-\ell(\ell+1)\sum_{i=j+1}^{k} h_i} \gs_{\ell h_j} X^1_{\ell
m}(j)\Biggr)
\cos(m \varphi)
\\
&&\hspace*{116pt}\quad\qquad {}+ \Biggl(\sum_{j=0}^k
e^{-\ell(\ell+1)\sum_{i=j+1}^{k}
h_i} \gs_{\ell h_j} X^2_{\ell m}(j)\Biggr)
\\
&&\hspace*{291pt}{}\times \sin(m\varphi) \Biggr) \Biggr)
\end{eqnarray*}
and that
\[
\sum_{j=0}^k e^{-\ell(\ell+1)\sum_{i=j+1}^{k} h_i}
\gs_{\ell h_j} X^i_{\ell m}(j)
\]
is a normally distributed random variable with mean zero and variance
\[
\bigl(e^{-\ell(\ell+1)\sum_{i=j+1}^{k} h_i} \gs_{\ell h_j}\bigr)^2 =
\frac{1}{2\ell(\ell+1)} \bigl(1 - e^{-2\ell(\ell+1)t_{k+1}}\bigr) = \gs
_{\ell t_{k+1}}^2.
\]
This implies that we have equality in law of
\begin{eqnarray*}
&& \sum_{j=0}^k e^{-\ell(\ell+1)\sum_{i=j+1}^{k} h_i}
\psi_\ell(j,y)
\\
&& \qquad= \sqrt{A_\ell} \gs_{\ell t_{k+1}} \Biggl(
X^1_{\ell0} L_{\ell0}(\vartheta)
\\
&&\hspace*{51pt}\quad\qquad {} + \sqrt{2} \sum_{m=1}^\ell
L_{\ell m}(\vartheta) \bigl( X^1_{\ell m} \cos(m
\varphi) + X^2_{\ell m} \sin(m\varphi) \bigr) \Biggr),
\end{eqnarray*}
where $((X^1_{\ell m}, X^2_{\ell m}), \ell\in\N_0, m=0,\ldots,\ell
)$ is a sequence of
independent, standard normally distributed random variables.

To implement the solution, we calculate $X_\ell$ exactly for
finitely many $\ell\in\N_0$ on a finite time and space grid.
One way to discretize the sphere is to take an equidistant grid
in $\vartheta\in[0,\pi]$ and $\varphi\in[0,2\pi)$. Then we add the
calculated $X_\ell$ and get an approximate solution; that is,
we simulate the approximate solution~$X^{\gk}$, $\gk\in\N_0$ by
\[
X^{\gk} = \sum_{\ell=0}^{\gk}
X_\ell
\]
on finitely many time and space points.
In what follows let us estimate the mean square
error when truncation of the series expansion at $\gk\in\N$ is done.

\begin{lemma}\label{lemconvstochheat}
Let $t \in\IT$ and $0=t_0 < \cdots< t_n = t$ be a discrete time
partition for
$n \in\N$, which yields a recursive representation of the
solution~$X$ of
equation~(\ref{eqstochheateqn}).
Furthermore assume that there exist
$\ell_0 \in\N$, $\ga> 0$ and a constant~$C>0$
such that the angular power spectrum $(A_\ell, \ell\in\N_0)$
satisfies $A_\ell\le C \cdot\ell^{-\ga}$ for all $\ell> \ell_0$.
Then the error of the approximate solution $X^\gk$ is bounded
uniformly in time and
independently of the time discretization by
\[
\bigl\llVert X(t) - X^\gk(t)\bigr\rrVert_{L^2(\gO;L^2(\IS^2))} \le\hat{C}
\cdot\gk^{-\ga/2}
\]
for all $\gk\ge\ell_0$, where
\[
\hat{C}^2 = \llVert X_0\rrVert_{L^2(\gO;L^2(\IS^2))}^2
+ C \cdot\biggl(\frac{2}{\ga} + \frac{1}{\ga+1} \biggr).
\]
\end{lemma}

\begin{pf}
Let $t \in\IT$ and $0=t_0 < \cdots< t_n = t$ be a partition of $[0,t]$
for some $n \in\N$.
Since $\E(\psi_\ell(j)) = 0$ for all $\ell\in\N_0$ and
$j=0,\ldots,n-1$, we first
observe that
%
\begin{eqnarray}
\label{eqerrorapproxstochheat}
&& \bigl\llVert X(t_n) - X^\gk(t_n)
\bigr\rrVert_{L^2(\gO;L^2(\IS^2))}^2\nonumber
\\
&&\qquad  = \Biggl\llVert\sum_{\ell=\gk+1}^\infty\sum
_{m=-\ell}^\ell e^{-\ell
(\ell+1)t_n}(X_0,Y_{\ell m})_{L^2(\IS^2)}
Y_{\ell m} \nonumber
\\
&&\hspace*{62pt}{} + \sum_{\ell=\gk+1}^\infty\sum
_{j=0}^{n-1} e^{-\ell(\ell+1)\sum_{i=j+1}^{n-1} h_i}
\psi_\ell(j)\Biggr\rrVert_{L^2(\gO;L^2(\IS^2))}^2
\\
&&\qquad  = \Biggl\llVert\sum_{\ell=\gk+1}^\infty\sum
_{m=-\ell}^\ell e^{-\ell
(\ell+1)t_n}(X_0,Y_{\ell m})_{L^2(\IS^2)}
Y_{\ell m}\Biggr\rrVert_{L^2(\gO
;L^2(\IS^2))}^2\nonumber
\\
&&\quad\qquad {}+ \Biggl\llVert\sum_{\ell=\gk+1}^\infty
\sum_{j=0}^{n-1} e^{-\ell(\ell+1)\sum_{i=j+1}^{n-1} h_i}
\psi_\ell(j)\Biggr\rrVert_{L^2(\gO;L^2(\IS^2))}^2.\nonumber
\end{eqnarray}
We define an isotropic Gaussian random field as in Lemma~\ref
{lemsimulatesequenceexpT} by
\begin{eqnarray*}
T&:=& \sum_{\ell=0}^\infty\sqrt{A_\ell}
\gs_{\ell t_{n}} \Biggl(  X^1_{\ell0} L_{\ell0}(
\vartheta)
\\
&&\hspace*{59pt}{} + \sqrt{2} \sum_{m=1}^\ell
L_{\ell m}(\vartheta) \bigl( X^1_{\ell m} \cos(m
\varphi) + X^2_{\ell m} \sin(m\varphi) \bigr) \Biggr)
\end{eqnarray*}
with angular power spectrum $(A_\ell\gs_{\ell t_n}^2, \ell\in\N_0)$
and denote similarly to Section~\ref{secapproxiGRF} by $T^\gk$ the
truncated series expansion.
Then
\begin{eqnarray*}
&& \Biggl\llVert\sum_{\ell=\gk+1}^\infty\sum
_{j=0}^{n-1} e^{-\ell(\ell+1)
\sum_{i=j+1}^{n-1} h_i}
\psi_\ell(j)\Biggr\rrVert_{L^2(\gO;L^2(\IS^2))}^2
\\
&&\qquad = \bigl\llVert
T - T^\gk\bigr\rrVert_{L^2(\gO;L^2(\IS^2))}^2.
\end{eqnarray*}
The angular power spectrum satisfies with the made assumptions that
\[
A_\ell\gs_{\ell t_n}^2 = A_\ell
\frac{1}{2\ell(\ell+1)}\bigl(1- e^{-2\ell(\ell+1)t_n}\bigr) \le C \ell
^{-\ga}
\ell^{-2} \cdot1 = C \ell^{-(\ga+ 2)}.
\]
With these parameters we apply Proposition~\ref{propconvtruncKLexpGRF}
to the difference of~$T$ and~$T^\gk$ which yields
\[
\bigl\llVert T - T^\gk\bigr\rrVert_{L^2(\gO;L^2(\IS^2))}^2
\le\hat{C}^2 \gk^{-\ga} = C \cdot\biggl(\frac{2}{\ga} +
\frac{1}{\ga+1} \biggr)\gk^{-\ga}.
\]
The first term on the right-hand side of the last equality of~(\ref
{eqerrorapproxstochheat}) is
bounded by
\begin{eqnarray*}
&& \Biggl\llVert\sum_{\ell=\gk+1}^\infty\sum
_{m=-\ell}^\ell e^{-\ell
(\ell+1)t_n}(X_0,Y_{\ell m})_{L^2(\IS^2)}
Y_{\ell m}\Biggr\rrVert_{L^2(\gO
;L^2(\IS^2))}^2
\\
&& \qquad= \sum_{\ell=\gk+1}^\infty\sum
_{m=-\ell}^\ell e^{-2\ell
(\ell+1)t_n}\bigl\llVert
(X_0,Y_{\ell m})_{L^2(\IS^2)} Y_{\ell m}\bigr
\rrVert_{L^2(\gO
;L^2(\IS^2))}^2
\\
&&\qquad \le e^{-2(\gk+1)(\gk+2)t_n} \llVert X_0\rrVert
_{L^2(\gO;L^2(\IS^2))}^2.
\end{eqnarray*}
This converges faster than any polynomial, and in particular,
it can be bounded by $C\gk^{-\ga}$, implying that
\[
\bigl\llVert X(t_n) - X^\gk(t_n)\bigr\rrVert
_{L^2(\gO;L^2(\IS^2))}^2 \le C \biggl( \biggl(\frac{2}{\ga} +
\frac{1}{\ga+1} \biggr)+ \llVert X_0\rrVert_{L^2(\gO;L^2(\IS^2))}^2
\biggr)\gk^{-\ga}.
\]
This completes the proof.
\end{pf}
We remark that it is not necessary that the angular power spectrum
$(A_\ell, \ell\in\N_0)$ of the $Q$-Wiener process decays with rate
$\ell^{-\ga}$ for $\ga>2$,
but that it is sufficient to assume that $\ga> 0$.
In an implementation in MATLAB we verified the theoretical results of
Lemma~\ref{lemconvstochheat}. We took as ``exact'' solution the
approximate solution at time $T=1$ with $\gk=2^7$ (since for larger
$\gk$ the elements of the angular power spectrum~$A_\ell$ and
therefore the increments were so small that MATLAB failed to calculate
the series expansion). We calculated the solution in one time step since
we have shown in Lemma~\ref{lemconvstochheat} that the convergence
rate is independent of the number of calculated time steps. Instead of
the $L^2(\IS^2)$ error in space, we used the maximum over all grid
points which is a stronger error. In Figure~\ref{figstochheatL2error}
the results and the theoretical convergence rates are shown for $\ga=
1,3,5$. One observes that the simulation results match the theoretical
results from Lemma~\ref{lemconvstochheat}.

%
\begin{figure}

\includegraphics{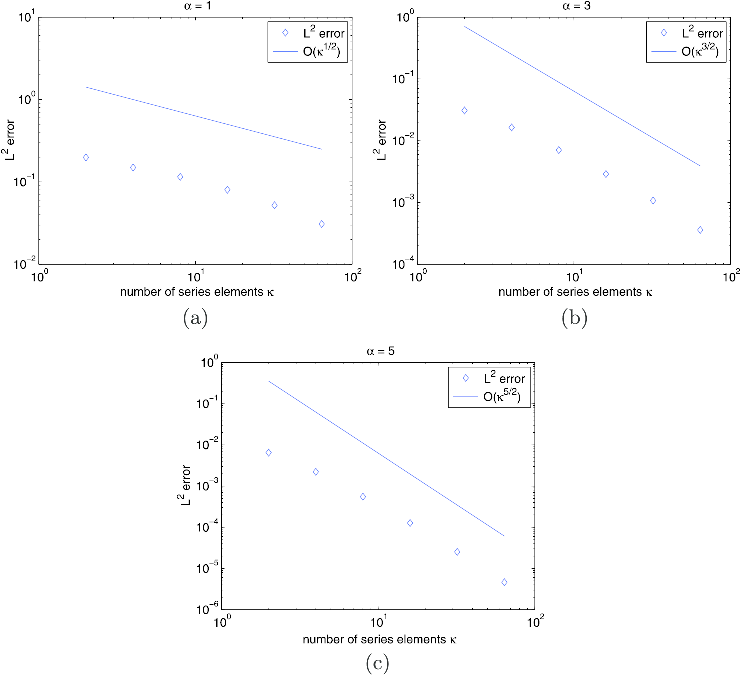}

\caption{Mean square error of the approximation of the stochastic heat
equation with different angular power spectra of the $Q$-Wiener process
and $100$~Monte Carlo samples.
\textup{(a)}~Angular power spectrum with parameter $\ga=1$.
\textup{(b)}~Angular power spectrum with parameter $\ga=3$.
\textup{(c)}~Angular power spectrum with parameter $\ga=5$.}\label{figstochheatL2error}
\end{figure}

Similarly to the proof of almost sure convergence of approximations of
isotropic Gaussian random fields in Section~\ref{secapproxiGRF}, we
need a $L^p$ convergence result for the approximation of the solution
of the stochastic heat equation to show pathwise convergence. This is
proven in
the following by a combination of Theorem~\ref{thmiGRFLpconv} and
Lemma~\ref{lemconvstochheat}.

\begin{lemma}\label{lemLpconvstochheat}
Let $t \in\IT$ and $0=t_0 < \cdots< t_n = t$ be a discrete time partition
for $n \in\N$, which yields a recursive representation of the
solution~$X$ of
equation~(\ref{eqstochheateqn}).
Furthermore assume that there exist $\ell_0 \in\N$, $\ga> 0$ and a
constant~$C>0$
such that the angular power spectrum $(A_\ell, \ell\in\N_0)$ satisfies
$A_\ell\le C \cdot\ell^{-\ga}$ for all $\ell> \ell_0$.
Then the error of the approximate solution $X^\gk$ is bounded uniformly
in time and independently of the time discretization by
\[
\bigl\llVert X(t) - X^\gk(t)\bigr\rrVert_{L^p(\gO;L^2(\IS^2))} \le
\hat{C}_p \cdot\gk^{-\ga/2}
\]
for all $p > 0$ and $\gk\ge\ell_0$, where $\hat{C}_p$ is a constant
that depends on the initial condition $\llVert X_0\rrVert _{L^{\max
(p,2)}(\gO
;L^2(\IS^2))}$, $p$, $C$ and $\ga$.
\end{lemma}

\begin{pf}
The result follows for $p \le2$ with Lemma~\ref{lemconvstochheat} and
with H\"older's inequality.
So we assume that $p>2$ from here on.
Let $t \in\IT$ and $0=t_0 < \cdots< t_n = t$ be a partition of $[0,t]$
for some $n \in\N$.
We first observe that
\begin{eqnarray*}
&& \bigl\llVert X(t_n) - X^\gk(t_n)\bigr
\rrVert_{L^p(\gO;L^2(\IS^2))}
\\
&& \qquad\le\Biggl\llVert\sum_{\ell=\gk+1}^\infty
\sum_{m=-\ell}^\ell e^{-\ell(\ell+1)t_n}(X_0,Y_{\ell m})_{L^2(\IS^2)}
Y_{\ell m}\Biggr\rrVert_{L^p(\gO;L^2(\IS^2))}
\\
&&\quad\qquad{} + \Biggl\llVert\sum_{\ell=\gk+1}^\infty
\sum_{j=0}^{n-1} e^{-\ell(\ell+1)\sum_{i=j+1}^{n-1} h_i}
\psi_\ell(j)\Biggr\rrVert_{L^p(\gO;L^2(\IS^2))}.
\end{eqnarray*}
Similarly to the proof of Lemma~\ref{lemconvstochheat}, the second
term is equal to the $L^p$ norm of the approximation error of an
isotropic Gaussian random field with angular power spectrum
$(A_\ell\gs_{\ell t_n}^2, \ell\in\N_0)$, which satisfies by
Theorem~\ref{thmiGRFLpconv} that
\begin{eqnarray*}
&& \Biggl\llVert\sum_{\ell=\gk+1}^\infty\sum
_{j=0}^{n-1} e^{-\ell(\ell
+1)\sum_{i=j+1}^{n-1} h_i}
\psi_\ell(j)\Biggr\rrVert_{L^p(\gO;L^2(\IS
^2))}
\\
&& \qquad= \bigl\llVert T-T^\gk\bigr\rrVert_{L^p(\gO;L^2(\IS^2))}
\\
&& \qquad\le(C_p)^{1/p} C^{1/2} \cdot\biggl(
\frac{2}{\ga-2} + \frac
{1}{\ga-1} \biggr)^{1/2}
\gk^{-\ga/2}.
\end{eqnarray*}
Furthermore the first term satisfies similarly to the proof of
Lemma~\ref{lemconvstochheat} that
\begin{eqnarray*}
&& \Biggl\llVert\sum_{\ell=\gk+1}^\infty\sum
_{m=-\ell}^\ell e^{-\ell
(\ell+1)t_n}(X_0,Y_{\ell m})_{L^2(\IS^2)}
Y_{\ell m}\Biggr\rrVert_{L^p(\gO
;L^2(\IS^2))}
\\
&& \qquad\le e^{-(\gk+1)(\gk+2)t_n} \llVert X_0\rrVert_{L^p(\gO;L^2(\IS^2))},
\end{eqnarray*}
which converges faster than any polynomial and therefore can be bounded by
$C \gk^{-\ga/2}$. This completes the proof.
\end{pf}

\begin{corollary}\label{corstochheateqnP-asconv}
Let $t \in\IT$ and $0=t_0 < \cdots< t_n = t$ be a discrete time
partition for $n \in\N$, which yields a recursive representation of
the solution~$X$ of equation~(\ref{eqstochheateqn}). Furthermore
assume that there exist $\ell_0 \in\N$, $\ga> 0$ and a constant~$C$
such that the angular power spectrum $(A_\ell, \ell\in\N_0)$
satisfies $A_\ell\le C \cdot\ell^{-\ga}$ for all $\ell> \ell_0$.
Then the error of the approximate solution $X^\gk$ is bounded
uniformly in time,
independently of the time discretization and asymptotically in~$\gk$ by
\[
\bigl\llVert X(t) - X^\gk(t)\bigr\rrVert_{L^2(\IS^2)} \le
\gk^{-\gb}, \qquad\IP\mbox{-a.s.}
\]
for all $\gb< \ga/2$.
\end{corollary}

\begin{pf}
The proof is similar to the one for isotropic Gaussian random fields in
Corollary~\ref{coriGRFP-asconv} but for completeness we include it here.
Let $\gb< \ga/2$. Then the Chebyshev inequality and Lemma~\ref
{lemLpconvstochheat} imply that
\begin{eqnarray*}
&&\IP\bigl(\bigl\llVert X(t)-X^\gk(t)\bigr\rrVert_{L^2(\IS^2)} \ge
\gk^{-\gb}\bigr)
\\
&&\qquad  \le\gk^{\gb p} \E\bigl(\bigl\llVert
X(t)-X^\gk(t)\bigr\rrVert_{L^2(\IS^2)}^p\bigr)
\\
&&\qquad  \le\hat{C}_p^p \gk^{(\gb-\ga/2)p}.
\end{eqnarray*}
For all $p > (\ga/2-\gb)^{-1}$, the series
\[
\sum_{\gk=1}^\infty\gk^{(\gb-\ga/2)p} < +
\infty
\]
converges, and therefore the Borel--Cantelli lemma implies the claim.\vadjust{\goodbreak}
\end{pf}

In Figure~\ref{figstochheatpatherror}, we show the corresponding error
plots to Figure~\ref{figstochheatL2error}, but instead of a Monte
Carlo simulation of the approximate $L^2(\gO;L^2(\IS^2))$ error, we
plot the error of one path of the stochastic heat equation. The
convergence results coincide with the theoretical results in
Corollary~\ref{corstochheateqnP-asconv}.

\begin{figure}

\includegraphics{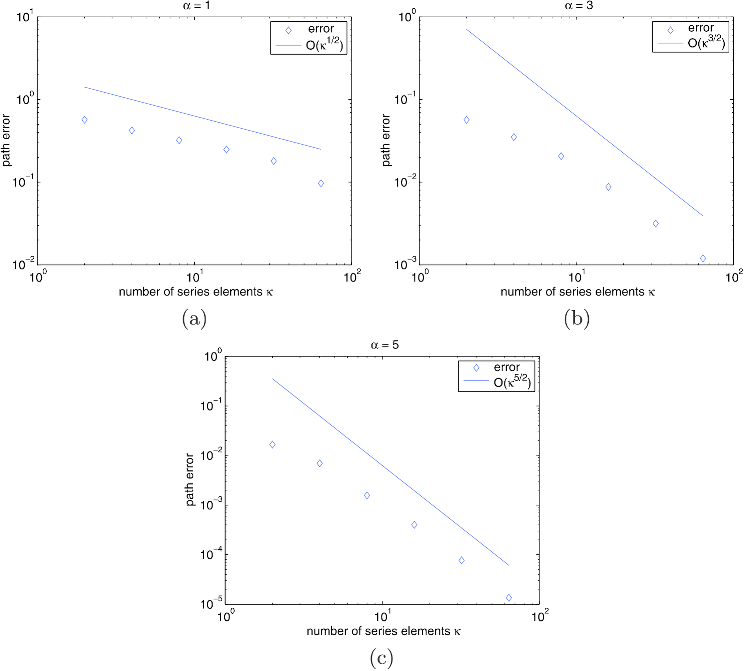}

\caption{Error of the approximation of a path of the stochastic heat
equation with different angular power spectra of the $Q$-Wiener process.
\textup{(a)}~Angular power spectrum with parameter $\ga=1$.
\textup{(b)}~Angular power spectrum with parameter $\ga=3$.
\textup{(c)}~Angular power spectrum with parameter $\ga=5$.}\label{figstochheatpatherror}
\end{figure}

\begin{appendix}
\section*{Appendix: Interpolation spaces}\label{appinterpolspaces}

In this \hyperref[appinterpolspaces]{Appendix} we give a more detailed introduction to interpolation
spaces than in Section~\ref{secPowSpecDec} and show 
that they are independent of the chosen interpolation couple.

We consider the sequence of spaces $(V^n(-1,1), n \in\N_0)$ that was
introduced in Section~\ref{secPowSpecDec} and
start now with the definition of fractional order spaces by the real
method of interpolation; see, for example, \citet{Triebel95}, Chapter~1.
We observe that for any two integers $k,n\in\N_0$ with $0\leq k < n$
the pair $(V^k(-1,1),V^n(-1,1))$ is an interpolation couple with
$V^n(-1,1) \subset V^k(-1,1) \subset L^2(-1,1)$.
For integers $k,m,n\in\N_0$ with $0\leq k < m < n$
so that $0<\theta:= (m-k)/(n-k) <1$,
we may therefore define the intermediate space
at ``fine-index'' $q\in[1,+\infty]$
\[
B^{m,(k,n)}_{2,q}(-1,1) = \bigl( V^k(-1,1),
V^n(-1,1) \bigr)_{\theta,q}
\]
by the real method of interpolation as introduced in \citet
{Triebel95}, Chapter~1.
Then these spaces
are equipped with the usual norms
$\llVert \cdot\rrVert _{B^{m,(k,n)}_{2,q}(-1,1)}$
given by
\[
\llVert u \rrVert_{B^{m}_{2,q}(-1,1)} = \cases{ \displaystyle\biggl(
\int_0^\infty t^{-\theta q} \bigl\llvert
K(t,u) \bigr\rrvert^q \,\frac{dt}{t} \biggr)^{1/q}, &
\quad for $1\leq q < + \infty$,
\vspace*{3pt}\cr
\displaystyle\sup_{t>0}
t^{-\theta} K(t,u), &\quad for $q = + \infty$,}
\]
where the $K$-functional is defined by
\[
K(t,u) = \inf_{u = v+w} \bigl( \llVert v \rrVert_{V^k(-1,1)}
+ t \llVert w \rrVert_{V^{n}(-1,1)} \bigr)
\]
for $t>0$.
We observe that in particular for every $n\in\N_0$
the pair of spaces
$(V^n(-1,1),V^{n+1}(-1,1))$ is
an interpolation couple.
Therefore, with $n \in\N_0$ and for $1\leq q \leq+\infty$,
we may extend the family
of exact interpolation spaces $(B^{m,(k,n)}_{2,q}(-1,1))_{0\le k < m <
n, q \in[1,+\infty]}$
also to\vspace*{2pt} noninteger numbers $s = n + \theta$, $\theta\in(0,1)$,
via
\[
B^{n+\theta}_{2,q}(-1,1):= \bigl(V^n(-1,1),
V^{n+1}(-1,1)\bigr)_{\theta,q}.
\]
Let us from here on simplify the notation and denote $V^n(-1,1)$
by~$V^n$ and $B^{m,(k,n)}_{2,q}(-1,1)$ by~$B^{m,(k,n)}_{2,q}$.
Our\vspace*{1pt} next proposition states that for $q = 2$ and $m \in\N$, the Besov
spaces $B^{m,(k,n)}_{2,2}$ are equal to $V^m$ for any choice $k<m<n$.

\begin{proposition}\label{propB=V}
Let $m \in\N$ be given. For any $k,n \in\N_0$ with $0 \le k < m <
n$, it holds that
$B^{m,(k,n)}_{2,2} = V^m$.
\end{proposition}

\begin{pf}
This result is classical; see, for example, \citet{Triebel95},
\citet{RunstSickel80} or \citet{SchmeisserTriebel87}, Chapter~6.5
and the references therein.
We present the detailed argument here for completeness.

By Theorem~\ref{thmBesovEquivNorm} we already know that the norm in
$V^m$ is equivalent to the weighted square summability of the
coefficients of the Fourier--Legendre expansion if $m \in\N_0$, so it
is sufficient to show the equivalence of the $B^{m,(k,n)}_{2,2}$-norm
and the convergence of the sum for all $0 \le k <m<n$.

Therefore we choose any $k,n \in\N_0$ with $0 \le k <m<n$ and $u \in
L^2(-1,1)$. Then $u$ admits the Fourier--Legendre expansion
\[
u = \sum_{\ell=0}^\infty u_\ell
\frac{2\ell+ 1}{2} P_\ell
\]
as has been seen above. Consider now $u \in V^m \cup B^{m,(k,n)}_{2,2}
\subset V^k$. We split $u$ into the sum $v+w$ with $v \in V^k$ and $w
\in V^n$ and the series expansions
\[
v = \sum_{\ell=0}^\infty(u_\ell-
w_\ell) \frac{2\ell+ 1}{2} P_\ell\quad\mbox{and}\quad w =
\sum_{\ell=0}^\infty w_\ell
\frac{2\ell+ 1}{2} P_\ell.
\]
Then Theorem~\ref{thmBesovEquivNorm} for integers implies that
\begin{eqnarray*}
K(t,u)^2 & \simeq&\inf_{u = v+w} \bigl(\llVert v
\rrVert_{V^k}^2 + t^2 \llVert w\rrVert
_{V^n}^2\bigr)
\\
& \simeq&\inf_{u = v+w} \sum_{\ell=0}^\infty
\frac{2\ell+1}{2} \bigl((u_\ell- w_\ell)^2
\bigl(1+\ell^{2k}\bigr) + w_\ell^2 \bigl(1+
\ell^{2n}\bigr) \bigr).
\end{eqnarray*}
We observe further that the infimum over all $u = v+w$ is equal to the
infimum over all square summable sequences $(w_\ell)_{\ell\in\N_0}
\in\ell^2(\N_0)$; that is,
\[
K(t,u)^2 \simeq\inf_{(w_\ell)_{\ell\in\N_0} \in\ell^2(\N_0)} \sum
_{\ell=0}^\infty\frac{2\ell+1}{2} G_\ell(u_\ell,w_\ell;t,k,n),
\]
where
\[
G_\ell(a,d;t,k,n):= (a-d)^2\bigl(1+\ell^{2k}
\bigr) + t^2 \,d^2\bigl(1+\ell^{2n}\bigr)
\]
is with respect to~$d \in\R$ a quadratic polynomial with positive
leading coefficient for all $\ell\in\N_0$. For $\ell\in\N_0$ its
minimum is attained at
\[
d_\ell:= \frac{a}{1+t^2 g_{kn}(\ell)},
\]
where
\[
g_{kn}(\ell):= \frac{1+\ell^{2n}}{1+\ell^{2k}}\geq1.
\]
This implies that
\begin{eqnarray*}
K(t,u)^2 & \simeq&\sum_{\ell=0}^\infty
\frac{2\ell+1}{2} \bigl((u_\ell- d_\ell)^2
\bigl(1+\ell^{2k}\bigr) + d_\ell^2 \bigl(1+
\ell^{2n}\bigr) \bigr)
\\
& = &\sum_{\ell=0}^\infty\frac{2\ell+1}{2}
u_\ell^2 \bigl(1+\ell^{2k}
\bigr)t^2 \frac{g_{kn}(\ell)}{1+t^2g_{kn}(\ell)}
\end{eqnarray*}
and leads with the definition of the norm and the theorem of
Fubini--Tonelli to
\begin{eqnarray*}
\llVert u\rrVert_{B^{m,(k,n)}_{2,2}}^2 & = &\int_0^\infty
t^{-2\theta} K(t,u)^2 \,\frac{dt}{t}
\\
& \simeq&\sum_{\ell=0}^\infty\frac{2\ell+1}{2}
u_\ell^2 \bigl(1+\ell^{2k}\bigr) \int
_0^\infty t^{-(2\theta+1)} \frac{t^2 g_{kn}(\ell
)}{1+t^2g_{kn}(\ell)} \,dt
\\
& = &\sum_{\ell=0}^\infty\frac{2\ell+1}{2}
u_\ell^2 \bigl(1+\ell^{2k}\bigr)
g_{kn}(\ell) \int_0^\infty
\frac{t^{1-2\theta}}{1+t^2g_{kn}(\ell)} \,dt,
\end{eqnarray*}
where $\theta:= (m-k)/(n-k) \in(0,1)$.
To complete the proof it remains to show that
\begin{eqnarray*}
&& \sum_{\ell=0}^\infty\frac{2\ell+1}{2}
u_\ell^2 \bigl(1+\ell^{2k}\bigr)
g_{kn}(\ell) \int_0^\infty
\frac{t^{1-2\theta}}{1+t^2g_{kn}(\ell)} \,dt
\\
&&\qquad \simeq\sum_{\ell=0}^\infty
u_\ell^2 \frac{2\ell+1}{2} \bigl(1+\ell^{2m}
\bigr)
\end{eqnarray*}
by the integer version of Theorem~\ref{thmBesovEquivNorm}; that is, we
have to prove the equivalence
\[
\bigl(1+\ell^{2k}\bigr) g_{kn}(\ell) \int
_0^\infty\frac{t^{1-2\theta}}{1+t^2g_{kn}(\ell)} \,dt \simeq1+
\ell^{2m} = 1 + \ell^{2((1-\theta)k + \theta n)}.
\]
Therefore let us split the integral first into
\begin{eqnarray*}
\int_0^\infty\frac{t^{1-2\theta}}{1+t^2g_{kn}(\ell)} \,dt &=& \int
_0^{g_{kn}(\ell)^{-1/2}} \frac{t^{1-2\theta
}}{1+t^2g_{kn}(\ell)} \,dt
\\
&&{} + \int
_{g_{kn}(\ell)^{-1/2}}^\infty\frac{t^{1-2\theta
}}{1+t^2g_{kn}(\ell)} \,dt
\end{eqnarray*}
and bound the two terms on the right-hand side from below and from
above by
\begin{eqnarray*}
\frac{1}{2} \int_0^{g_{kn}(\ell)^{-1/2}}
t^{1-2\theta} \,dt & \le&\int_0^{g_{kn}(\ell)^{-1/2}}
\frac{t^{1-2\theta
}}{1+t^2g_{kn}(\ell)} \,dt
\\
&\le& \int_0^{g_{kn}(\ell)^{-1/2}}
t^{1-2\theta} \,dt
\\
& = &\frac{1}{2-2\theta} g_{kn}(\ell)^{\theta- 1}
\end{eqnarray*}
and
\begin{eqnarray*}
\frac{1}{2g_{kn}(\ell)}\int_{g_{kn}(\ell)^{-1/2}}^\infty
t^{-1-2\theta} \,dt & \le&\int_{g_{kn}(\ell)^{-1/2}}^\infty
\frac{t^{1-2\theta
}}{1+t^2g_{kn}(\ell)} \,dt
\\
& \le&\frac{1}{g_{kn}(\ell)}\int_{g_{kn}(\ell)^{-1/2}}^\infty
t^{-1-2\theta} \,dt
\\
&=& \frac{1}{2\theta} g_{kn}(\ell)^{\theta- 1}.
\end{eqnarray*}
This implies overall that
\begin{eqnarray*}
\frac{1}{4(1-\theta)\theta} g_{kn}(\ell)^{\theta- 1} &\le&\int
_0^\infty\frac{t^{1-2\theta}}{1+t^2g_{kn}(\ell)} \,dt
\\
&\le&
\frac{1}{2(1-\theta)\theta} g_{kn}(\ell)^{\theta- 1}
\end{eqnarray*}
and moreover that
\begin{eqnarray*}
\bigl(1+\ell^{2k}\bigr) g_{kn}(\ell) \int
_0^\infty\frac{t^{1-2\theta}}{1+t^2g_{kn}(\ell)} \,dt &\simeq&\bigl(1+
\ell^{2k}\bigr) g_{kn}(\ell)^\theta
\\
&=& \bigl(1+
\ell^{2k}\bigr)^{1-\theta} \bigl(1 + \ell^{2n}
\bigr)^\theta.
\end{eqnarray*}
We observe that the function $x^p$, $p \in(0,1)$ is concave on $\R_+$
and satisfies $(x+y)^p \ge2^{p-1} (x^p + y^p)$, which implies finally that
\begin{eqnarray*}
\bigl(1+\ell^{2k}\bigr)^{1-\theta} \bigl(1 + \ell^{2n}
\bigr)^\theta&\simeq&\bigl(1+\ell^{2(1-\theta)k}\bigr) \bigl(1 +
\ell^{2\theta n}\bigr)
\\
&\simeq&1+\ell^{2((1-\theta)k + \theta n)}
\\
&= &1 +
\ell^{2m}.
\end{eqnarray*}
This completes the proof.
\end{pf}

Based on Proposition~\ref{propB=V}, it is clear that one can use for every
$m\in\N$ in place of $B^{m,(k,n)}_{2,2}$ simply $V^m$.
Moreover, for fractional $\eta= n +\theta$ with $n\in\N_0$ and
some $0<\theta<1$, we write also $V^\eta$ in place of $B^\eta_{2,2}$.
\end{appendix}

\section*{Acknowledgments}
The authors acknowledge Roman Andreev, Sonja Cox, Markus Hansen,
Sebastian Klein,
Markus Knopf and Hans Triebel for fruitful discussions and helpful comments.
Furthermore we thank an anonymous referee for the very useful comments
that helped to improve and generalize the results.


%

\printaddresses
\end{document}